\documentclass[11pt,reqno]{amsart}

\usepackage[hmargin=4cm,vmargin=4cm]{geometry}

\usepackage{amsthm,amssymb}

\usepackage{graphicx}

\usepackage{amscd,amssymb}
\usepackage[leqno]{amsmath}
\usepackage{tikz}
\usepackage{tikz-3dplot}
\usepackage{tikz-cd}
\usepackage{faktor}

\newtheorem  {theorem}                  {Theorem}
\newtheorem* {theorem*}                   {Theorem}

\newtheorem {lemma}[theorem] {Lemma}
\newtheorem {prop}[theorem]      {Proposition}
\newtheorem* {prop*}     {Proposition}

\newtheorem {corollary}[theorem]      {Corollary}

\theoremstyle{definition}
\newtheorem {defi}[theorem] {Definition}
\newtheorem {Remark} [theorem]         {Remark}

\newtheorem* {Example*}    {Example}

\def\R{\mathbb{R}}

\newcommand{\pp}[2]{\frac{\partial#1}{\partial#2}}
\newcommand{\norm}[1]{\|#1\|}

\newcommand{\ech}{\operatorname{ECH}}

\DeclareMathOperator\curl{curl}

\definecolor{myg}{RGB}{60, 125, 60}

\title{Contact type solutions and non-mixing of the 3D Euler equations}

\author{Robert Cardona}
\address{Robert Cardona,
Departament de Matemàtiques i Informàtica, Universitat de Barcelona, Gran Via de les Corts Catalanes 585, 08007 Barcelona, Spain \newline \it{e-mail: robert.cardona@ub.edu }
 }

\author{Francisco Torres de Lizaur}
\address{Francisco Torres de Lizaur, Departamento de An\'alisis Matem\'atico \& IMUS, Universidad de Sevilla, C/ Tarfia s/n, 41012 Sevilla, Spain \newline \it{e-mail: ftorres2@us.es}}

\thanks{RC acknowledges financial support from the Universitat de Barcelona, from the Margarita Salas postdoctoral
contract financed by the European Union-NextGenerationEU and its host institutions: the Universitat Politècnica de Catalunya and the Instituto de Ciencias Matemáticas. RC was partially supported by the AEI grant PID2019-103849GB-I00 / AEI / 10.13039/501100011033 and the AGAUR grant 2021 SGR 00603. FTL was supported by the European Union's H2020 program under the Marie Curie grant
 GDSFLOWS-101063565, and by the Spanish MINCIN through
the Ram\'on y Cajal program. Both authors also acknowledge partial support from the grant “Computational, dynamical and geometrical
complexity in fluid dynamics”, Ayudas Fundaci\'on BBVA a Proyectos de Investigaci\'on Cient\'ifica 2021.}

\begin{document}

\begin{abstract}
We prove that on any closed Riemannian three-manifold $(M,g)$ the time-dependent Euler equations are non-mixing on the space of smooth volume-preserving vector fields endowed with the $C^1$-topology, for any fixed helicity and large enough energy, solving a problem posed by Khesin, Misiolek, and Shnirelman. To prove this, we introduce a new framework that assigns contact/symplectic geometry invariants to large sets of time-dependent solutions to the Euler equations on any 3-manifold with an arbitrary fixed metric. This greatly broadens the scope of contact topological methods in hydrodynamics, which so far have had applications only for stationary solutions and without fixing the ambient metric. We further use this framework to prove that spectral invariants obtained from Floer theory, concretely embedded contact homology, define new non-trivial continuous first integrals of the Euler equations in certain regions of the phase space endowed with the $C^{1,s}$-topology, producing countably many disjoint invariant open sets.
\end{abstract}

\maketitle

\section{Introduction}

On a Riemannian manifold $(M,g)$, an ideal fluid (i.e. incompressible and without viscosity) moves according to the Euler equations
 \begin{equation}\label{eq:Eul}
\begin{cases}
\frac{\partial}{\partial t} u + \nabla_u u &= -\nabla p\,, \\
\operatorname{div}u=0\,,
\end{cases}
\end{equation}
where $p$ is a scalar function called the pressure, and the vector field $u$ is the fluid's velocity field, which is a non-autonomous vector field on $M$. The term $\nabla_u u$ denotes the covariant derivative of $u$ along itself, and $\operatorname{div}$ is the divergence associated with the Riemannian metric $g$. The Euler equations define a (local) flow $\varphi_t$ on the configuration space, which is the set of $C^{\infty}$-divergence-free vector fields $\mathfrak{X}_{\mu}(M)$. For a given initial condition $u_0$, the flow is known to be defined locally around $t=0$ \cite{EM}, and it is a well-known open problem whether $\varphi_t$ is actually a complete flow in dimensions higher than two. If one chooses as configuration space the set of $C^k$ divergence-free vector fields, where $k>1$ is a non-integer, the flow is also locally well defined.

The dynamical properties of the time-dependent Euler equations are hence captured by this flow on an infinite-dimensional space. In dimension 2, since the flow is complete for regular enough velocities \cite{WO}, these properties have been more thoroughly studied; for instance, the existence of wandering solutions (orbits of $\varphi_t$ that, after a certain positive time, never come back to a vicinity of the initial condition). Nadirashvili \cite{Na} constructed such solutions in the annulus, more precisely, he found a $C^{1}$-open set $U$ of smooth velocity fields such that $\varphi_t(U) \cap U =\emptyset$ after a certain positive time. From a slightly different perspective, Shnirelman \cite{Sh} showed that wandering solutions exist in the two-torus, not in velocity space but in the full lagrangian phase space of diffeomorphisms.

In higher dimensions, the unknown completeness of the flow should not preclude us from studying finer dynamical properties: this is not only possible but of great interest, as recent results attest \cite{CMPP, TL, EPT2, KY}. Khesin, Kuksin, and Peralta-Salas  \cite{KKPS} showed that the Euler equation on compact 3-manifolds is non-mixing, in the following sense: there exist open sets of volume-preserving vector fields (either smooth or $C^k$ with $k>4$ non-integer) whose evolution under $\varphi_t$ stay at a bounded distance in the $C^{k}$-topology for $k> 4$, from other well-chosen open sets. Later on, the same authors proved that this property holds (in the $C^{k}$-topology for $k> 10$) for a generic set of initial conditions, smooth or $C^k$ with $k>10$, near certain classes of steady states, called non-degenerate shear flows \cite{KKP2}. A key ingredient in the work of these authors is the KAM theorem, which lies behind the $C^{4}$ regularity threshold in their results. An open problem proposed in \cite[Problem 31]{KMS} is to show that analogous results hold in lower regularity spaces. \\

The first contribution of this work is establishing analogous non-mixing properties for the Euler equations defined on smooth (or $C^k$ with $k>1$ non-integer) divergence-free vector fields endowed with the $C^1$-topology on any Riemannian closed three-manifold $(M,g)$. Our result is in some sense sharp, since, loosely speaking, in the vorticity formulation it corresponds to $C^0$-topology. 

To make our statement precise, we fix the three known invariants of the Euler equation that are continuous in the $C^{1}$-topology: the energy, the helicity, and the homotopy class of the vorticity (which, in this work, will always be assumed to be everywhere non-vanishing and thus defines a homotopy class of non-vanishing vector fields). Notice that if we consider possibly vanishing vorticities or non-vanishing vorticities on different homotopy classes of non-vanishing fields, non-mixing follows from different elementary arguments, see Remark \ref{rem:vanishvort}. Hence, let $\mathcal{V}_{a,h,e}$ be the set of (either $C^{\infty}$ or $C^{k}$, $k>1$ non-integer) divergence-free vector fields in $M$ preserving the induced Riemannian volume $\mu$ with fixed helicity $h$, fixed energy $e$ and whose curl lies in the homotopy class of non-vanishing vector fields $a$. 
\begin{theorem}\label{thm:main}
    Let $(M,g)$ be a closed Riemannian three-manifold. For any helicity $h\neq 0$, any homotopy class of non-vanishing vector fields $a$, and any energy $e \geq e_{0}$ (where $e_{0}$ depends on $h$ and $a$), there exist two $C^1$-open sets $C_{a,h,e}$ and $N_{a,h,e}$ of $\mathcal{V}_{a,h,e}$ such that
    $$\varphi_t(C_{a,h,e})\cap N_{a,h,e}=\emptyset,$$
    for each $t$ where the local flow defined by the Euler equation $\varphi_t:\mathfrak{X}_\mu(M) \longrightarrow \mathfrak{X}_\mu(M)$ is defined.
\end{theorem}

The proof of Theorem \ref{thm:main}, as in \cite{KKPS}, uses Helmholtz's transport of vorticity and an analysis of the space of vorticities. The strategy is to find non-trivial functionals on the space of smooth velocities that are invariant under the flow $\varphi_t$ by assigning to each velocity some property of the corresponding vorticity, and then showing that these are continuous with respect to the $C^{1}$ topology. In the case of \cite{KKPS}, the functional is defined as the volume occupied by knotted invariant tori of the vorticity, which is why KAM is needed to ensure continuity at certain vector fields with integrable regions, and this continuity only holds in the $C^{4}$ topology. Instead of relying on KAM theory, we bring ideas from symplectic and contact geometry to define invariant $C^1$-open properties of time-dependent solutions of the Euler equations.

Thus, our work falls within the scope of contact hydrodynamics (i.e. contact geometry methods in hydrodynamics), inaugurated in the seminal work of Etnyre and Ghrist \cite{EG,EG1,EG2} for the study of stationary solutions to the Euler equations, a connection suggested too by D. Sullivan (see also \cite{GK}). However, until now contact hydrodynamics had applications exclusively to the study of solutions that do not depend on time, and always in closed three-manifolds with an adapted (not fixed a priori) ambient Riemannian metric. See  \cite{CMPP, CV2, MOPS, PSRT, PSS} for some works related to this approach. Our work, instead, introduces a new framework that broadens this scope and addresses the most relevant setting: time-dependent solutions instead of stationary ones, and on any three-manifold with a fixed (and arbitrary) ambient metric. As a byproduct of the techniques used in the proof of Theorem \ref{thm:main}, it can also be established that strong (or even rotational) Beltrami fields, a very important class of stationary solutions to the Euler equations, are not $C^0$-dense among stationary solutions in the round three-sphere (Corollary \ref{cor:statdens}).\\

The main connection between the time-dependent Euler equations and contact geometry that we introduce to prove Theorem \ref{thm:main} is the notion of contact type solution (Definition \ref{def:contactsolution}). We prove that the set of initial conditions that define contact type solutions is, for any $a,h$ and large enough $e$, a large $C^1$-open set $ C_{a, h, e} \subset \mathcal{V}_{a,h,e}$ that is invariant by the Euler equations. The other main piece of the proof is to ensure that these open sets are not dense, that this, that there exist $C^{1}$-open sets of non-contact type solutions in $\mathcal{V}_{a,h,e}$ that never intersect $C_{a,h,e}$.

Once the notion of contact type solution has been introduced, a whole new collection of invariants for the Euler equation can be defined, by assigning contact geometric invariants to solutions in a natural way. These can be simple invariants like the isotopy class of a contact structure (Lemma \ref{lem:continv}), or more sophisticated ones obtained via Floer theory and pseudoholomorphic curves. One of the most relevant and successful theories in Floer and Seibert-Witten theory is embedded contact homology (ECH), introduced by Hutchings (see \cite{Hut2} and references therein). We show that spectral invariants arising in ECH \cite{Hut} can be used to define continuous integrals (with suitable topologies) of the Euler equations (Theorem \ref{thm:contspec}), which can often be shown to be non-trivial, establishing a deep connection between the time-dependent Euler equations and modern technology in contact topology. These integrals provide a better picture of contact type regions of the phase space of the Euler equations on Riemannian three-manifolds: as an application, we establish the existence of countably many $C^{1,s}$-open sets that never intersect each other under the evolution of the Euler equations. 

Recall first that any homotopy class of non-vanishing vector fields $a$ uniquely determines a homotopy class $p_a$ of plane fields, given by any plane field whose direct sum with an element in $a$ spans $TM$.  To state our second main result, let us mention that the proof of Theorem \ref{thm:main} involves as a step showing that in $\mathcal{V}_{a,h,e}$ (for any $a$, any $h$, and any large enough $e$), any isotopy class of contact structure $[\xi]$ whose homotopy class is in $p_a$ defines a (non-empty) $C^1$-open subset $\mathcal{V}_{[\xi]} \subset \mathcal{V}_{a,h,e}$ that is invariant by the Euler equations. 

\begin{theorem}\label{thm:spec}
    Let $(M,g)$ be a closed Riemannian three-manifold. Let $a$ be a homotopy class of non-vanishing vector fields such that any (and hence every) plane field in $p_a$ has a torsion Euler class. For any helicity $h$, large enough energy $e$ and any isotopy class of contact structures $[\xi]$ in $p_a$ , the following hold:
    \begin{enumerate}
        \item Equip $\mathcal{V}_{[\xi]}$ with the topology induced by the $C^{1,s}$-topology in $\mathfrak{X}_\mu(M)$, for any $s\in (0,1)$. There exists a non-trivial continuous functional
    $$c:\mathcal{V}_{[\xi]} \longrightarrow \mathbb{R},$$ 
    that is a first integral of the Euler equations,
    \item there are countably many $C^{1,s}$-open sets $U_i$ in $\mathcal{V}_{[\xi]}$  such that for any $i\neq j$, we have
    $$ \varphi_t(U_i)\cap U_j=\emptyset, $$
    for every $t$ for which the flow of the Euler equations is defined.    
    \end{enumerate}  
\end{theorem}

On any closed three-manifold, there are always countably many homotopy classes $a$ such that any plane field in $p_a$ has a torsion Euler class. It is also well known that any homotopy class of plane fields admits a contact structure. Hence, Theorem \ref{thm:spec} applies to countably many choices of $a$ on any closed three-manifold, see Remark \ref{rem:chernhyp}. On the three-sphere, or more generally any rational homology sphere, every plane field has a vanishing Euler class so the theorem applies to each $\mathcal{V}_{[\xi]}$ of any $\mathcal{V}_{a,h,e}$ with large enough energy $e$.\\

The paper is organized as follows: Section \ref{s:intro} reviews the classical functionals on the space of divergence-free fields that are invariant under the Euler equation, and states the main continuity properties of the curl operator and its inverse that we will use. Section \ref{s:open} introduces the set of contact type exact divergence-free fields, some $C^0$-topological properties of this set, and the notion of contact type solution. We introduce a criterion to find exact divergence-free fields that are not of contact type. In Section \ref{s:data} we crucially show that one can always construct large sets of initial conditions that are (or not) of contact type on any Riemannian 3-manifold with any fixed values of helicity $h$, homotopy class $a$ of the vorticity field, and sufficiently large energy $e$. This paves the way for the proof of Theorem \ref{thm:main} in Section \ref{s:main}. We also prove in this Section the important Corollary \ref{cor:statdens} concerning non $C^{0}$-density of rotational and strong Beltrami fields in the round three-sphere. To conclude, we recall the definition of the ECH spectral invariants in Section \ref{s:ech} and use our framework to prove Theorem \ref{thm:spec}, as a combination of Theorems \ref{thm:contspec} and \ref{thm:C1smixing}.\\

\noindent \textbf{Acknowledgements.} The authors are grateful to Daniel Peralta-Salas for useful discussions. We thank Boris Khesin and Rohil Prasad for their useful comments on the first draft of this work.

\section{Divergence-free vector fields and classical invariants}\label{s:intro}
Throughout this paper, we will work on a closed orientable three-dimensional manifold $M$ equipped with a Riemannian metric $g$, and we denote by $\mu$ the induced Riemannian volume form.

\subsection{Divergence-free fields on Riemannian manifolds}
Denote by $\mathfrak{X}_\mu(M)$ the space of smooth vector fields preserving $\mu$, that is $X\in \mathfrak{X}_\mu(M)$ if and only if $\mathcal{L}_X\mu=0$. This condition is also equivalent to the fact that $\iota_X\mu$ is a closed two-form. Denote by $\mathfrak{X}_\mu^0(M)$ the space of exact divergence-free vector fields: i.e. those vector fields $X$ such that $\iota_X\mu$ is exact. 

 Given $X\in \mathfrak{X}_\mu(M)$ we define the energy of $X$ as
$$ \mathcal{E}(X)=\int_M g(X,X)\mu, $$
and the helicity of $X\in \mathfrak{X}^{0}_\mu(M)$ as
$$ \mathcal{H}(X)= \int_M \alpha\wedge d\alpha, $$
where $\iota_X\mu=d\alpha$. It is well known that the helicity of a vector field is preserved under volume-preserving diffeomorphisms. In other words, if $\varphi:M\rightarrow M$ satisfies $\varphi^*\mu=\mu$, then
$$ \mathcal{H}(\varphi_*X)=\mathcal{H}(X), $$
and it is actually the only integral invariant of an exact divergence-free vector field (see \cite{EPT, KPY}).\\

For a vector field $X\in \mathfrak{X}(M)$, we denote by $X^\flat$ the one-form dual to $X$ by the metric, which is determined by imposing $g(X,Y)=X^\flat(Y)$ for each $Y\in \mathfrak{X}(M)$. The curl operator
$$ \operatorname{curl}: \mathfrak{X}(M) \longrightarrow \mathfrak{X}_\mu^0(M)  $$
is a linear differential operator that assigns to a vector field $X$ the vector field $Y$ determined by the equation $\iota_Y\mu=d(X^\flat)$. The helicity of $Y$ can then be written in terms of $X$, since $\mathcal{H}(Y)=\int_M X^\flat \wedge dX^\flat$. The vector field $Y=\operatorname{curl}(X)$ is also called the \emph{vorticity field} of $X$. We state here the well-known properties of the curl operator that we will need (see e.g. \cite[Lemma 2.3]{KKPS}).

\begin{lemma}\label{lem:curl}
    The curl operator on smooth vector fields defines a continuous surjective map, where we endow $\mathfrak{X}(M)$ with the $C^{k+1}$-topology and $\mathfrak{X}^{0}_{\mu}(M)$ with the $C^{k}$-topology, with $k\geq 0$.
\end{lemma}
On exact divergence-free vector fields the curl operator admits a well-defined inverse
\begin{align*}
    \curl^{-1}:\mathfrak{X}_\mu^0(M) &\longrightarrow \mathfrak{X}_\mu^0(M)\\
        X(p) & \longmapsto \int_M k(p,q)X(q)d\mu(q),
\end{align*}
where $k(p,q)$ is a matrix-valued integral kernel.  One can use the inverse curl operator to give another expression of the helicity
$$ \mathcal{H}(X) = \int_M X\cdot \curl^{-1}X\mu,$$
where the dot denotes the scalar product defined by the metric $g$, i.e. $X\cdot Y= g(X,Y)$.
By its definition as an integral operator, one deduces the following continuity property.
\begin{lemma}\label{lem:curlinverse}
 The operator $\curl^{-1}$ is continuous in $\mathfrak{X}_\mu^0(M)$ equipped with the $C^0$-topology.
\end{lemma}

\subsection{Primitive bounds of exact two-forms}
We will introduce in this subsection two auxiliary lemmas that we will need for different arguments throughout the rest of the sections.

\begin{lemma}\label{lem:C0bound}
Let $\eta$ be an exact two-form on a closed three-manifold $M$. For any $s\in (0,1)$, there exists a primitive $\beta$ of $\eta$ such that
$$ \norm{\beta}_{C^{0,s}} \leq C(s) \norm{\eta}_{C^0}. $$
\end{lemma}

\begin{proof}
Let $Y$ be the vector field satisfying $i_{Y} \Omega=\eta$, and let $X:=\curl^{-1} Y$. Then the 1-form $\beta=i_{X} g$ is a primitive of $\eta$, and moreover it satisfies $d \star \beta=0$, because $X$ is divergence-free. Observe also that by virtue of Lemma \ref{lem:curlinverse}, $||\beta||_{C^{0}} \leq C ||\eta||_{C^{0}}$.

Now let $\delta=-\star d \star$ be the $L^{2}$-adjoint of $d$ acting on 1-forms. By standard elliptic estimates, we have
\begin{align*}
||\beta||_{W^{1, p}} \leq C \bigg(||d \beta+\delta \beta||_{L^{p}}+||\beta||_{L^{p}}\bigg) \\ \leq 
C ||\eta||_{L^{p}}+C ||\beta||_{L^{p}} \leq C' ||\eta||_{C^{0}}
\end{align*}
where we have used that both $d \beta=\eta$ and $\beta$ are small in $C^{0}$, and thus in $L^{p}$, and $\delta \beta=-\star d \star \beta=0$. 

By the Sobolev embedding theorem, $W^{1, p}(M) \subset C^{r, \alpha}(M)$ whenever  $p>3$ and $r+\alpha=1-\frac{3}{p}$, so we conclude that
\[
||\beta||_{C^{0, 1-\frac{3}{p}}} \leq C(p) ||\eta||_{C^{0}}
\]
Setting $s=1-3/p$ yields the claim.
\end{proof}

On the other hand, if an exact two-form is small in the $C^{0,s}$ norm, we can find a primitive that is small in the $C^{1,s}$-norm by using the Green operator of the Laplace-Beltrami operator, see for example \cite[Lemma 3.28]{CV1}.

\begin{lemma}\label{lem:C1bound}
Let $\eta$ be an exact two-form in a closed three-manifold $M$. For any $s\in (0,1)$, there exists a primitive $\beta$ of $\eta$ such that
$$ \norm{\beta}_{C^{1,s}} \leq C(s) \norm{\eta}_{C^{0,s}}. $$
\end{lemma}

\subsection{The classical invariants of the Euler equations: energy, helicity, homotopy}\label{ss:eul}
Let $u_t$ be a time-dependent solution to the Euler equations \eqref{eq:Eul} in $(M,g)$. The first conserved quantity of $u_t$ is its energy: the equation
$$ \mathcal{E}(u_t)=\mathcal{E}(u_0)$$
holds for each $t$ for which the solution is defined.  Another main property of the Euler equations \eqref{eq:Eul} is the so-called Helmholtz's transport of vorticity, also known as Kelvin's circulation theorem. Taking the curl of Eq. \eqref{eq:Eul}, we obtain the equation
\begin{equation} \label{eq:Helm}
\partial_t\omega+[u,\omega]=0, 
\end{equation}
where $\omega:=\operatorname{curl}(u)$. This equation means that the vorticity field of $u$ is transported by the flow, i.e. the vorticity of $u$ at time $t$ is diffeomorphic to the vorticity at time zero for each time. Concretely, if $u_t$ is a solution to the Euler equations, define the time-dependent ODE
\begin{align*}
    \dot{x}=u_t(x), \enspace x\in M,
\end{align*}
which describes the motion of each particle in time (this is the so-called Lagrangian description of fluid motion). If $\phi_t:M\rightarrow M$ is the flow integrating this equation, then 
\begin{equation}\label{eq:transport}
    \curl u_t= (\phi_t)_*(\curl u_0),
\end{equation}
and hence at each time, the vorticity is diffeomorphic to the initial one.

Another interpretation can be given by considering the Euler equation as a geodesic flow on the Lie group of volume preserving diffeomorphisms. Denote by $\operatorname{Diff}_\mu^0(M)$ the set of diffeomorphisms of $M$ isotopic to the identity and preserving the volume form $\mu$.  This is a Lie group whose Lie algebra $\mathfrak{g}$ is the space of divergence-free vector fields in $M$. Then the vorticity transport implies that the curl of a solution to the Euler equations belongs to the adjoint orbit of $\curl u_0$ (under the action of $\operatorname{Diff}_\mu^0(M)$) for all time. 

Since the helicity of a vector field is invariant under the action of $\operatorname{Diff}_\mu^0(M)$, it follows that the helicity of the curl of $u_t$ is an invariant of $u_t$ as well. Finally, observe that if the curl of $u_0$ is non-vanishing, then it defines a homotopy class of non-vanishing vector fields. The fact that $\curl u_t$ is obtained by pushing forward $\curl u_0$ by a diffeomorphism isotopic to the identity trivially implies that this homotopy class is another invariant of $u_t$. In this work, we will study the dynamical system defined by the Euler equations on the invariant set defined by fixing a positive energy value of the solution, a helicity value of its vorticity, and a homotopy class of non-vanishing vector fields (realized by the vorticity of the solutions that we consider).

\begin{Remark}\label{rem:vanishvort}
If we do not restrict to non-vanishing vorticities and a fixed homotopy class, then non-mixing in the $C^1$-topology follows from different elementary arguments combined with the techniques in Section \ref{s:data}. For instance, construct initial conditions with non-homotopic vorticities and their neighborhood will not mix under the Euler equations. One can also consider vorticities with a different number of hyperbolic zeroes (see \cite[Remark 6.7]{KKPS}), etc. Of course, such elementary arguments do not allow proving Theorem \ref{thm:main}, the much stronger existence of new conserved quantities (Theorem \ref{thm:spec}), and do not give tools to analyze the neighborhood of non-vanishing stationary solutions (e.g. Corollary \ref{cor:KKPSint}). 
\end{Remark}

\section{Open properties of vorticities}\label{s:open}

In this section, we import ideas from contact and symplectic geometry to find properties of a non-vanishing exact divergence-free vector field that can be robust to $C^0$-perturbations, invariant under volume-preserving transformations, and do not depend on the helicity. In this section, every vector field that we consider is non-vanishing everywhere in $M$.

\subsection{Contact type vorticities}\label{ss:ctvorticities}

Recall that a one-form $\alpha$ in $M$ is a contact form if $\alpha\wedge d\alpha \neq 0$. If $M$ is oriented, for instance by the Riemannian volume form $\mu$, we say that $\alpha$ is a positive (respectively negative) contact form if the orientation induced by $\alpha\wedge d\alpha$ matches the orientation of $M$ (respectively the opposite one). The plane distribution $\xi=\ker \alpha$ defined by $\alpha$ is called a positive (respectively negative) contact structure, and any positive multiple of $\alpha$ is another contact form defining $\xi$. A contact form $\alpha$ defines a Reeb vector field $R$ by the equations 
\begin{equation*}
    \begin{cases}
        \alpha(R)=1,\\
        \iota_Rd\alpha=0.
    \end{cases}
\end{equation*}
Two contact structures $\xi_0,\xi_1$ are homotopic if there exists a one-parametric family of contact structures $\xi_t$ connecting them. If there exists a diffeomorphism $\phi:M\rightarrow M$ such that $\phi^*\xi_1=\xi_0$, then the contact structures are isomorphic, and if $\phi$ is defined via an isotopy $\phi_t$, we say that the contact structures are isotopic. It is clear that isotopic contact structures are homotopic, and it is a classical result that the converse holds on closed manifolds, this is known as Gray stability.

\begin{theorem}[Gray \cite{Gray}]\label{thm:gray}
    Let $M$ be a closed manifold. Two homotopic contact structures are isotopic.
\end{theorem}

We say that an exact non-degenerate two-form $\omega$ in a three-dimensional manifold $M$ is of contact type \cite{Mc} if $\omega=d\alpha$ for some contact form $\alpha$. This definition readily generalizes to odd-dimensional manifolds as well. The same terminology can be introduced for exact divergence-free vector fields.

\begin{defi}
Let $X$ be an exact divergence-free vector field in $M$. If $\iota_X\mu=d\alpha$ for some (positive or negative) contact form $\alpha$, we say that $X$ is (positive or negative) of contact type.
\end{defi}

Although we will not need this observation, note that if $X$ is of contact type then it is a reparametrization of the Reeb field defined by the contact form $\alpha$, concretely $X=\alpha(X)R$. We denote by $\mathcal{C}_\mu \subset \mathfrak{X}_\mu^0(M)$ the set of contact type exact divergence-free vector fields, which is given by the union of $\mathcal{C}_\mu^+(M)$ and $\mathcal{C}_\mu^-(M)$ (positive and negative contact type). We will see shortly that these two sets are in fact disjoint.

Denote by $\operatorname{Cont}^+(M)$ the set of positive contact structures in $M$, defined up to isotopy. For each $[\xi] \in \operatorname{Cont}^+(M)$, denote by $\mathcal{C}_\mu^+(M,\xi)$ the set of exact divergence-free vector fields $X$ such that $\iota_X\mu=d\alpha$ where $\alpha$ defines a contact structure isotopic to $\xi$. With this notation, the set of (positive) contact type divergence-free vector fields can be expressed as
\begin{equation}\label{eq:contdec}
    \mathcal{C}_\mu^+(M)=\bigcup_{\xi\in \operatorname{Cont}^+(M)} \mathcal{C}_\mu^+(M,\xi).
\end{equation} 
These sets have several useful properties as subsets of $\mathfrak{X}_\mu^0(M)$ equipped with the $C^0$-topology.
\begin{lemma}\label{lem:Xcont}
The following properties hold:
\begin{itemize}
    \item[-] $\mathcal{C}_\mu^+$ and $\mathcal{C}_\mu^-$ are $C^0$-open in $\mathfrak{X}_\mu^0(M)$,
    \item[-] each $\mathcal{C}_\mu^+(M,\xi)$ is $C^0$-open in $\mathcal{C}_\mu^+$,
    \item[-] if $\xi$ and $\xi'$ are not isotopic, then $\mathcal{C}_\mu^+(M,\xi)\cap \mathcal{C}_\mu^+(M,\xi')=\emptyset$.
\end{itemize}
An analogous statement holds for negative contact type fields, and shows in particular that $\mathcal{C}_\mu^+(M)$ and $\mathcal{C}_\mu^-(M)$ are disjoint.
\end{lemma}
Hence, the decomposition \eqref{eq:contdec} is disjoint, i.e. $\mathcal{C}_\mu^+(M)=\bigsqcup_{\xi\in \operatorname{Cont}^+(M)} \mathcal{C}_\mu^+(M,\xi)$. The two-form $\iota_X\mu$ has several primitives that might or not be contact forms, but each contact primitive defines the same contact structure up to isotopy.
\begin{proof}
Let $X$ be an exact divergence-free vector field in $ \mathcal{C}_\mu^+$, and $Y$ a $C^0$-close exact divergence-free vector field. Hence $\iota_X\mu=d\alpha$ for some positive contact form $\alpha$, and writing $\iota_Y\mu=d\beta$ we know that
$$ \norm{d\alpha-d\beta}_{C^0}<\delta. $$ 
The two-form $d\alpha-d\beta$ is an exact $C^0$-small two-form, so by Lemma \ref{lem:C0bound}, we construct a $C^0$-small one-form $\gamma$ such that $d\gamma=d(\alpha-\beta)$. It follows that $\eta=\alpha-\gamma$ satisfies $d\eta=d\beta$ and $\eta\wedge d\eta$ is $C^0$-close to $\alpha\wedge d\alpha$. Hence if $\delta$ is small enough, the one-form $\eta$ satisfies $\eta\wedge d\eta>0$, and since $\iota_Y\mu=d\eta$ we deduce that $Y\in C^+_\mu$.\\

For $\delta$ small enough the linear interpolation between $\alpha$ and $\eta$ is a path of contact forms. Indeed $\alpha_t= (1-t) \alpha + t\eta$ satisfies
\begin{align*}
    \alpha_t\wedge d\alpha_t&= (1-t)^2\alpha\wedge d\alpha + t(1-t)\alpha\wedge d\eta+ (1-t)t\eta \wedge d\alpha + t^2 \alpha \wedge d\alpha,
\end{align*}
and since $\eta$ is $C^0$-close to $\alpha$ and $d\eta$ is $C^0$-close to $d\eta$, each term above is a non-negative multiple of $\alpha\wedge d\alpha$. We deduce that $\alpha_t\wedge d\alpha_t>0$, which shows that $\ker \alpha$ and $\ker \eta$ are homotopic through contact structures. Gray stability then shows that these contact structures are isotopic, thus proving the second statement. 

For the third statement, assume that $\iota_X\mu=d\alpha=d\tilde \alpha$ for two different contact forms $\alpha,\tilde \alpha$. Then we know that $X$ lies in the one-dimensional kernel of $d\alpha$ and $d\tilde \alpha$, and that $\alpha(X)>0$ and $\tilde \alpha(X)\neq 0$. On the other hand, the helicity does not depend on the primitive so $\int_M \tilde \alpha \wedge d\alpha >0$, and thus $\tilde \alpha(X)>0$. In particular, the three-forms $\alpha\wedge d\alpha$, $\alpha\wedge d\tilde \alpha$, $\tilde \alpha \wedge d\alpha$ and $\tilde \alpha \wedge d\tilde\alpha$ are all volume forms defining the same orientation. Thus we can argue exactly as we did before to show that the linear interpolation $\alpha_t= (1-t)\alpha + t\tilde \alpha$ defines a path of contact forms. This shows that $\ker \alpha$ and $\ker \tilde \alpha$ are homotopic through contact forms, and hence contact isotopic by Gray's stability. This shows if $Y\in C_\mu^+(M,\xi)\cap C_\mu^+(M,\xi')$, then $\xi$ and $\xi'$ are in the same isotopy class of contact structures.
\end{proof}

Observe that every vector field in $\mathcal{C}_\mu^{+}(M)$ has positive helicity, and every vector field $\mathcal{C}_\mu^{-}(M)$ has negative helicity. In subsection \ref{ss:obs}, we will introduce more dynamical properties of an exact divergence-free vector field that can be useful to find other $C^0$-open sets in $\mathfrak{X}_\mu^0(M)$.

\subsection{Contact type solutions of the Euler equation}\label{ss:contacttype}

Observe that if $X \in \mathcal{C}^{+}_{\mu}(M, \xi)$, then for any volume preserving diffeomorphism $F: M \rightarrow M$ isotopic to the identity, we have that $F_{*} X \in \mathcal{C}^{+}_{\mu}(M, \xi)$ (and analogously for negative contact type fields). Indeed, we have that $\iota_X\mu=d\alpha$ for a contact form $\alpha$ whose kernel is the contact structure $\xi$, so
\begin{align}\label{eq:contactinvariance}
\begin{split}
    \iota_{F_{*} X}\mu&=F^*(\iota_{Y}\mu)\\
    &=d F^*\alpha 
    \end{split}
\end{align} and $F^*\alpha$ is a contact form for the contact structure $F^{*} \xi$, which is in the same isotopy class as $\xi$. 

Therefore, if an initial velocity field $u_{0}$ has vorticity $\curl u_{0} \in \mathcal{C}_{\mu}(M, \xi)$, the time-dependent solution of the Euler equation $u_t$ with initial condition $u_0$ will satisfy at any time that $\curl u_t \in \mathcal{C}_{\mu}(M, \xi)$ because of Helmholtz's transport of vorticity (Eq. \eqref{eq:transport}). This motivates the following definition.

\begin{defi}\label{def:contactsolution}
    A solution $u_t$ to the Euler equations in $(M,g)$ whose curl is non-vanishing is called of \textit{contact type} if its curl is of contact type. 
\end{defi}

Analogously, we say that an initial condition $u_0\in \mathfrak{X}_\mu(M)$ is of contact type if the curl of $u_0$ is of contact type. 

\begin{Remark}\label{rem:rotBelt}
Notice that if $\alpha=g(u_0,\cdot)$ is a contact form, then $u_0$ is a contact type initial condition, but that in general a contact type initial condition $u_0$ will not satisfy that $\alpha$ is precisely one of the contact primitives of $d\alpha$. If $X$ is a stationary solution that happens to be a rotational Beltrami field, i.e. $\curl X= fX$ with $f\in C^\infty(M)$ a non-vanishing function, then $g(X,\cdot)$ is indeed a contact form. This includes curl eigenvectors with a non-zero eigenvalue. Thus, rotational Beltrami fields are a particular case of (stationary) contact type solutions of the Euler equations.
\end{Remark}

In the next subsection, we introduce a criterion that ensures that a divergence-free field is neither in $\mathcal{C}^{+}_{\mu}(M)$ nor in $\mathcal{C}_\mu^-(M)$. A key property of this criterion is its stability under $C^{0}$ small perturbations, which will be instrumental in Section \ref{s:main} to find $C^{1}$-open sets of solutions to the Euler equation that are not of contact type.

%Notice that, by Lemma \ref{lem:Xcont}, contact type initial conditions are an open set with respect to the $C^{1}$-topology. 

%\textcolor{myg}{We will denote by $\mathcal{V}_{a, e, h, [\xi] }$ the set of initial conditions $u_0$ with $L^{2}$-norm (i.e energy) $e$, $\mathcal{H}(\curl u_0)=h$, $[\curl u_{0}]=a$ and $\curl u_{0} \in \mathcal{C}_{\mu}(M, \xi)$. The above discussion ensures that whenever an initial condition belongs to $\mathcal{V}_{a, e, h, [\xi] }$, so does the solution $u_{t}$ for any time of existence. It remains to check, of course, that these sets are in fact non-empty. The next section shows that this is the case in any Riemannian manifold, provided the energy $e$ is large enough.}

\subsection{Cieliebak's criterion}\label{ss:obs}

Geometric currents are a very useful tool to characterize the geometric properties of differential forms \cite{Su}. A $k$-current is a continuous linear function defined on $\Omega^k(M)$, and we denote by $\mathcal{Z}^k(M)$ the space of $k$-currents, that we endow with the weak$^{*}$ topology. There exists a continuous boundary operator 
$$\partial: \mathcal{Z}^k(M) \rightarrow \mathcal{Z}^{k-1}(M),$$
defined by $\partial c(\omega)=c(d\omega)$, where $c$ is a $k$-current, $\omega$ is a $(k-1)$-form and $d$ is the standard exterior derivative of forms. Thanks to the boundary operator, which satisfies $\partial^2=0$, one can speak of exact and closed currents.

Given a non-vanishing vector field $X$ and a point $p\in M$, the Dirac $1$-current $\delta^{X}_p\in \mathcal{Z}^1(M)$ evaluates a form on $X|_p$, i.e.
$$\delta^{X}_p(\beta)=\beta(X)|_p, \enspace \beta \in \Omega^1(M).$$ A $1$-current $z\in \mathcal{Z}^1(M)$ is called a foliation current of $X$ if it lies in the closed convex cone generated by its Dirac $1$-currents. A foliation cycle refers to a closed foliation $1$-current, which could be exact as well. \\

Using this language, one can restate a characterization of contact type two-forms due to McDuff \cite{Mc} and Sullivan \cite{Su} in our context of exact divergence-free vector fields.

\begin{theorem}[\cite{Mc}]
Let $X$ be a non-vanishing vector field preserving a volume form $\mu$ and such that $\iota_X\mu$ is exact. Then $X\in \mathcal{C}_\mu$ if and only if given any primitive $\alpha$ of $\iota_X\mu$ and any exact foliation cycle $z$ of $X$ we have $z(\alpha)\neq 0$.
\end{theorem}

A particularly interesting example of a foliation cycle is given by a periodic orbit of $X$, oriented by the positive direction of the flow. Indeed, such an orbit $\gamma$ defines a $1$-current by integration $\gamma(\alpha)=\int_{\gamma}\alpha$.
We deduce the following simple criterion, used by Cieliebak in \cite{Cie}.

\begin{corollary}\label{cor:nonR}
Let $X\in \mathfrak{X}_\mu^0(M)$ be an exact divergence-free vector field. If $X$ has two null-homologous closed orbits $\gamma_1, \gamma_2$ such that
$$ \int_{\gamma_1}\alpha>0, \qquad \int_{\gamma_2}\alpha<0,$$
for some $\alpha$ such that $\iota_X\mu=d\alpha$, then $X\not \in \mathcal{C}_\mu$.
\end{corollary}

Indeed, if there exist such periodic orbits and primitive $\alpha$, a linear combination of $\gamma_1$ and $\gamma_2$ defines an exact foliation cycle such that $z(\alpha)=0$. In \cite{Cie}, Cieliebak proved that there exist embeddings of $S^3$ in $(\mathbb{R}^4,\omega_{std})$ such that a section of the kernel of the induced two-form on $S^3$ satisfies the conditions above, and that any $C^0$-perturbation of the embedding induces a two-form also satisfying these properties. In other words, the embedded hypersurface cannot be $C^0$-approximated by contact hypersurfaces (hypersurfaces such that the two-form induced by the ambient symplectic form is of contact type). Note that a $C^0$-perturbation of the embedding can drastically change the characteristic foliation given by the kernel of the two-form, so a priori dynamical methods are not enough and methods from symplectic geometry, like Floer theory, are required for his proof. One of the key properties of the closed orbits that he considers is that they are non-degenerate. Recall that a periodic orbit is non-degenerate if no eigenvalue of the linearized Poincar\'e map of the periodic orbit is a root of the unity. \\

 What we use in this work is an adaptation of Cieliebak's obstruction for perturbations in the space of exact divergence free flows with respect to a fixed volume-form $\mu$. The criterion turns out to be $C^0$-robust in the space of exact flows on any closed three-manifold. In our setting, the perturbation is small at a dynamical level (the flow) and dynamical arguments are sufficient for the proof.
 
\begin{theorem}\label{thm:C0obs}
Let $X$ be as in Corollary \ref{cor:nonR} such that $\gamma_1$ and $\gamma_2$ are non-degenerate periodic orbits of $X$. Then there exists $\delta_0>0$ such that any vector field $Y\in \mathfrak{X}_\mu^0$ satisfying
\begin{equation}
\norm{X-Y}_{C^0} < \delta_0.
\end{equation}
satisfies $Y\not \in \mathcal{C}_\mu$.
\end{theorem}
\begin{proof}
For some small $\delta>0$, let $Y$ be any vector field such that 
\begin{equation}\label{eq:vfs}
\norm{X-Y}_{C^0} < \delta.
\end{equation}
Denote by $\varphi_X^t:M \rightarrow M$ and $\varphi_Y^t:M \rightarrow M$ the flows defined by $X$ and $Y$ respectively. We will first show that for $\delta$ small enough, the vector field $Y$ admits closed orbits $\tilde \gamma_1, \tilde \gamma_2$ that are arbitrarily $C^1$-close to $\gamma_1$ and $\gamma_2$ respectively.\\

Choose a tubular neighborhood $U$ of $\gamma_1$, and $D'$ a disk transverse to $X$ centered at a point $x_0\in \gamma_1$. There exists a smaller disk $D\subset D'$ such that the first-return map of $X$ on $D$ is well-defined and with image in $D'$. Denote by $\tau:D\longrightarrow \mathbb{R}$ the first-return time at a point $x\in D$, and by choosing $D$ small enough we assume that $\varphi_X^t(x)\subset U$ for all $t\in [0,\tau(x)]$, for each $x\in D$. Fixing any small enough $\varepsilon>0$, we can choose $\delta$ smaller than some $\delta_0$ such that the flow of $Y$ satisfies as well that the first-return map is well defined in $D$, with image in $D'$ and if $\tilde \tau: D \longrightarrow \mathbb{R}$ denotes the first return time of the flow of $Y$, we have
$$ |\tilde \tau(x) - \tau(x)| < \varepsilon, \quad \text{for each } x\in D. $$
Let $s(t):\mathbb{R}_{\geq 0} \longrightarrow M$ be the solution to the ODE defined by $X$ with initial condition $s(0)=x_0$. Similarly, let $\tilde s(t)$ be the solution to the ODE defined by $Y$ with initial condition $x_0$. Triangle inequality yields
\begin{equation}
\frac{d}{dt} \norm{\tilde s(t)- s(t)}_{C^0} \leq \norm{X(\tilde s(t))-X(s(t))}_{C^0} + \delta.
\end{equation} 
Let $K>0$ be a uniform upper bound of $\norm{X}_{C^0}$ in $U$, then by the mean value theorem
\begin{equation}
\frac{d}{dt} \norm{\tilde s(t)-s(t)}_{C^0} < K\norm{\tilde s(t)-s(t)}_{C^0}+\delta.
\end{equation}
We apply Gronwall's inequality and deduce
\begin{equation}
\norm{\tilde s(t)-s(t)}_{C^0} < \frac{\delta}{K} (e^{K.t}-1),
\end{equation}
hence
\begin{equation}
\norm{\tilde s(t)-s(t)}_{C^0} < \frac{\delta}{K} (e^{K.(\tau(x_0)+\varepsilon)}-1), \text{ for } t\in [0,\tau+\varepsilon].
\end{equation}
Given some small $\delta'$, we can choose $\delta < \delta' \frac{K}{3(e^{K.(\tau(x_0)+\varepsilon)}-1)}$, then 
\begin{equation}\label{eq:intcurves}
\norm{\tilde s(t)-s(t)}_{C^0} < \frac{\delta'}{3}, \text{ for } t\in [0,\tau(x_0)+\varepsilon].
\end{equation}
We can estimate the $C^1$-distance between $s$ and $\tilde s$ as
\begin{align}\label{eq:C1norm}
\begin{split}
\norm{\tilde s(t)-s(t)}_{C^1} &= \norm{\tilde s(t)-s(t)}_{C^0} + \norm{Y(\tilde s(t)) - X(s(t))}_{C^0}\\
&\leq \norm{\tilde s(t)-s(t)}_{C^0} + \norm{Y(\tilde s(t)) - X(\tilde s(t))}_{C^0} \\
&+ \norm{X(\tilde s(t)) - X( s(t))}_{C^0}.
\end{split}
\end{align}
On the other hand, by continuity, there is some $C>0$ such that 
\begin{equation}\label{eq:contX}
\norm{X_p-X_q}_{C^0}<C \norm{p-q}_{C^0}
\end{equation}
 for all $p,q \in M$. We might have chosen $\delta'< \frac{\tilde{\delta}}{C}$, for any given positive small $\tilde \delta$, and we can assume as well that we chose $\delta<\frac{\tilde \delta}{3}$. Using Equations \eqref{eq:vfs},  \eqref{eq:intcurves} and \eqref{eq:contX} in Equation \eqref{eq:C1norm} we deduce that for $t\in [0,\tau(x_0)+\varepsilon)$
 \begin{align}
 \norm{\tilde s(t)- s(t)}_{C^1} &< \frac{\delta'}{3} + \delta + \frac{\tilde \delta}{3}\\
 &< \frac{\tilde \delta}{3} + \frac{\tilde \delta}{3} + \frac{\tilde \delta}{3} \\ 
 &<\tilde \delta. \label{eq:boundIntcurv}
 \end{align}
We will apply the previous computations to a certain trajectory of $Y$. The first-return map of $X$ admits an isolated fixed point $x_0$ whose degree is non-zero by assumption since $\gamma_1$ is a non-degenerate periodic orbit of $X$. By \cite[Theorem 8.4.4]{KH}, we know that given a disk in $\hat D \subset D$ of small radius $\hat \delta$ (the radius depends on $\delta$ and can be made arbitrarily small by decreasing $\delta$) centered at $x_0$, the first return map of $Y$, which is $C^0$-close to the first return map of $X$ by Equation \eqref{eq:boundIntcurv}, has some fixed point $\tilde x_0$ in $\hat D$. In other words, the integral curve $\hat s(t)$ of $Y$ with initial condition $\tilde x_0$ is periodic with period $\tilde \tau(\tilde x_0)$, which can be assumed to be smaller than $\tau(x_0)+\varepsilon$. The point $\tilde x_0\in D$ is $\delta'$-close to $x_0$, hence if $\delta$ is small enough continuity with respect to initial conditions shows that
\begin{equation}
\norm{\hat s(t) - s(t)}_{C^1} < \tilde \delta, \quad \text{for } t\in [0,\tau(x_0)+\varepsilon],
\end{equation}
for an arbitrarily small $\tilde \delta$. Since both $s(t)$ and $\hat s(t)$ are periodic with period smaller than $\tau(x_0)+\varepsilon$, we deduce that the closed orbit $\tilde \gamma_1$ of $Y$ through $\tilde x_0$ is $\tilde \delta$-close to $\gamma_1$ in the $C^1$-topology. The same argument applies to $\gamma_2$, so if $\delta$ is small enough we have proven the existence of closed orbits $\tilde \gamma_1$ and $\tilde \gamma_2$ of $Y$ that are $C^1$-close to $\gamma_1$ and $\gamma_2$ respectively.\\

By hypothesis, both $X$ and $Y$ are exact divergence-free, which implies that $\iota_X\mu$ and $\iota_Y\mu$ are exact. Let us argue that $\iota_Y\mu$ admits a primitive that is $C^0$-close to a primitive of $\iota_X\mu$. Given some $\epsilon>0$, if $\delta$ is small enough then we know that
\begin{equation}
\norm{\iota_Y\mu -\iota_X\mu}_{C^0} <\epsilon.
\end{equation}
Hence writing $\iota_X\mu=d\alpha$ and $\iota_Y\mu=d\beta$, the previous inequality can be written as
\begin{equation}
\norm{d\beta - d\alpha}_{C^0} < \epsilon.
\end{equation}
This means that $d\beta-d\alpha$ is an exact $C^0$-small two-form, so applying Lemma \ref{lem:C0bound} we deduce that it admits a $C^0$-small primitive $\eta$. Then $\iota_Y\mu= d(\alpha+\eta)$ and by taking $\delta$ small enough the closed orbits $\tilde \gamma_1$ and $\tilde \gamma_2$ are very close to $\gamma_1$ and $\gamma_2$ in the $C^1$-norm, and the one-form $\alpha+\eta$ is $C^0$-small. We can find some $\delta_0>0$ such that if $\delta<\delta_0$, we have
\begin{equation}
\int_{\tilde \gamma_1 } \alpha+\eta >0, \quad \int_{\tilde \gamma_2} \alpha + \eta <0,
\end{equation}
since $\int_{\gamma_1}\alpha>0$ and $\int_{\gamma_2}\alpha<0$.
Corollary \ref{cor:nonR} implies that $d(\alpha+\eta)$ is not of contact type and thus  $Y\not \in \mathcal{C}_\mu$.
\end{proof} 
This criterion will be instrumental in finding $C^0$-open sets of exact divergence-free fields that are not of contact type.

\section{Initial conditions with prescribed data}\label{s:data}

Let $(M,g)$ be a closed Riemannian three-manifold. In this section, we will construct volume-preserving vector fields whose curl has some prescribed properties and that have any large enough value of the energy.

\subsection{Vorticity fields with prescribed helicity}
First, we will construct exact divergence-free vector fields that are of contact type or robustly not of contact type (as in Theorem \ref{thm:C0obs}) in any homotopy class and with any given non-zero value of the helicity. These will be, potentially, the vorticities of the initial data that we want to consider in the Euler equations. Since we only deal with helicity, through this section we only require fixing a volume form in $M$. First, we construct contact type exact divergence-free vector fields with given helicity on each homotopy class of non-vanishing vector fields. 

\begin{lemma}\label{lem:Hreeb}
Let $M$ be a closed three-manifold equipped with a volume form $\mu$. For any homotopy class of non-vanishing vector fields and real number $h\neq 0$, there is an exact divergence-free vector field $Y$ of contact type such that $\mathcal{H}(Y)=h$.
\end{lemma}
\begin{proof}
Choose a homotopy class of vector fields $[u]$. Equip $M$ with an auxiliary Riemannian metric, then the orthogonal complement of $u$ defines a cooriented plane field $\xi_0$. By the works of Lutz and Martinet \cite{Lut,Mar}, there is a homotopy of plane fields $\xi_t, t\in [0,1]$ such that $\xi=\xi_1$ is a contact structure. Let $\alpha$ be a defining one-form of $\xi$, i.e. $\ker \alpha=\xi$. The Reeb field $R$ defined by $\alpha$ is trivially homotopic to $u$: the homotopy $\xi_t$ induces a homotopy $u_t$, by considering a smooth family of sections of the orthogonal bundle of $\xi_t$, and such that $u_0=u$. Since $u_1$ and $R$ are both transverse to $\xi$, they are trivially homotopic as non-vanishing vector fields.  The vector field $Y$ defined by $\iota_Y\mu=d\alpha$ is a reparametrization of $R$, and hence is homotopic to $u$ as well. On the other hand, the helicity of $Y$ is equal to the contact volume 
$$\mathcal{H}(Y)=\int_M \alpha\wedge d\alpha.$$ Thus, by considering the two-form $Cd\alpha$ for an arbitrary non-vanishing constant $C$ we obtain a contact type divergence-free vector field $Y_C$, defined by $\iota_{Y_C}\mu=d(C\alpha)$, whose helicity is 
$$\mathcal{H}(Y_C)=C^2\int_M \alpha\wedge d\alpha.$$
The whole argument works as well by considering a negative contact structure, i.e. one for which $\alpha\wedge d\alpha$ induces the opposite orientation than $\mu$. Hence, choosing an appropriate $C$ we might prescribe any non-zero value of the helicity $\mathcal{H}(Y)$.
\end{proof}

We now establish an analogous result for vector fields satisfying the robust obstruction introduced in Section \ref{ss:obs}.

\begin{prop}\label{prop:Hnonreeb}
Fix some volume form $\mu$ on a closed three-manifold and a real number $h\in \mathbb{R}$. On each homotopy class of non-vanishing vector fields, there exists $X\in \mathfrak{X}_\mu^0$ having two non-degenerate null-homologous closed orbits $\gamma_1,\gamma_2$ such that $\int_{\gamma_1}\alpha>0$ and $\int_{\gamma_2}\alpha<0$, where 
 $\alpha$ is such that $\iota_X\mu=d\alpha$, and such that $\mathcal{H}(X)=h$.
\end{prop}

\begin{proof}
Arguing as in the proof of Lemma \ref{lem:Hreeb}, there exists a contact structure $\xi$, given by the kernel of a global one-form $\lambda$, in any given homotopy class of plane fields. In particular, any Reeb field of $\xi$ belongs to a prescribed homotopy class of non-vanishing vector fields. Choose two null-homologous embedded closed curves $\gamma_1',\gamma_2'$ on $M$. These can be $C^0$-perturbed to embedded closed curves positively transverse to $\xi$, see e.g. \cite[Theorem 3.3.1]{Ge}, that we denote by $\gamma_1$ and $\gamma_2$.

Near each $\gamma_i$ there is a neighborhood $U_i\cong D^2_{\delta_i}\times S^1$, where $D^2_{\delta_i}$ is a disk of radius $\delta_i$, endowed with coordinates $(r_i,\varphi_i,\theta_i)$ where the contact structure is defined by the form
$$ \alpha_i= d\theta_i + \frac{r_i^2}{2}d\varphi_i, $$
see e.g. \cite[Section 4.1]{Ge}.
The coordinates $(r_i,\varphi_i)$ are polar coordinates in the disk factor, while $\theta_i$ is the canonical angle coordinate of $S^1$. Since locally $\alpha_i$ and $\lambda$ both define the contact structure $\xi$, it follows that $\alpha_i=f_i \lambda$ for some positive function $f_i\in C^\infty(U)$. Choose a positive function $g$ that is equal to $1$ near $r=\delta_i$ and equal to $f_i$ near $\delta_i-\varepsilon$ for some very small $\varepsilon$. Up to slightly shrinking each $U_i$, this shows that there is a global contact form $\alpha$ such that $\alpha|_{U_i}=\alpha_i$. For the moment, we fix the volume form $\mu'=\alpha \wedge d\alpha$. For this volume form, the Reeb field $R$ of $\alpha$ satisfies $\iota_R \mu'=d\alpha$.\\

We proceed to modify $\alpha$ near $U_1$. In this neighborhood, the volume form $\mu'$ is equal to $r_1 dr_1\wedge d\varphi_1 \wedge d\theta_1$. Consider
$$ \eta_1= \phi(r_1)d\theta_1 + \frac{r_1^2}{2}d\varphi_1, $$
where $\phi:[0,\delta_1)\rightarrow \mathbb{R}$ is a smooth positive function satisfying $\phi(r_1)= 1+\tau \frac{r_1^2}{2}$ near $r=0$, with $\tau\in \mathbb{R}\setminus \mathbb{Q}$ being some irrational number, and $\phi(r_1)=1$ near $r_1=\delta_1$. We have
$$ d\eta_1= \phi'(r_1)dr_1\wedge d\theta_1 + r_1 dr_1\wedge d\varphi_1, $$
and the vector field $Y_1$ defined by $\iota_{Y_1}\mu'= d\eta_1$ is
$$ Y_1= \pp{}{\theta_1} - \frac{\phi'(r_1)}{r_1} \pp{}{\varphi_1}, $$
This vector field extends as $R$ away from $U_1$, because $\eta_1=\alpha$ near $r=\delta_1$. It is clearly homotopic to $R$ since $Y_1$ is homotopic to $\pp{}{\theta_1}$ relative to the boundary, so the homotopy class has not changed after this first modification.  Furthermore $\gamma_1=\{0\}\times S^1$ is a non-degenerate periodic orbit of $Y_1$ satisfying $\int_{\gamma_1}\eta_1=2\pi$.\\

We modify $\alpha_2$ in $U_2$ by considering 
$$ \eta_2= \psi(r_2)d\theta_2 + \frac{r_2^2}{2}d\varphi_2, $$
where $\psi:[0,\delta_2)\rightarrow \mathbb{R}$ is a smooth function equal to $-1+\tau \frac{r^2_2}{2}$ near $r_2=0$ and equal to $1$ near $r_2=\delta_2$. Then the vector field $Y_2$ defined by $\iota_{Y_2}\mu'=d\eta_2$ is 
$$ Y_2= \pp{}{\theta_2} + \frac{\psi'(r_2)}{r_2}\pp{}{\varphi_2}. $$
As before, the closed curve $\gamma_2$ is a non-degenerate periodic orbit of $Y_2$, but in this case $\int_{\gamma_2}\eta_2=-2\pi$. Since each $\eta_i$ coincides with $\alpha$ near $r_i=\delta_i$, there is a global one-form $\eta$ equal to $\alpha$ away from $U_i$ and equal to $\eta_i$ in $U_i$. The vector field $Y$ defined by $\iota_Y \mu'=d\eta$ has $\gamma_1$ and $\gamma_2$ as non-degenerate closed orbits and $\int_{\gamma_1}\eta>0$ and $\int_{\gamma_2}\eta<0$. Again $Y$ is homotopic to $Y_1$, since in $U_2$ it is homotopic to $\pp{}{\theta_2}$ relative to the boundary, and thus the homotopy class has not changed after this modification.\\

 It remains to show that we can further modify $\eta$ to prescribe the helicity of the vector field determined by the differential of $\eta$. We argue as before and assume that there is a third neighborhood $U_3\cong D_{\delta_3}\times S^1$ (away from $U_1\cup U_2$) of some closed curve transverse to $\xi$ where we can assume that $$\eta|_{U_3}=\alpha|_{U_3}=d\theta_3+\frac{r_3^2}{2}d\varphi_3.$$
We modify $\eta$ to a one-form with an expression of the form
$$ \tilde \eta= F(r_3) d\theta_3 + G(r_3)d\varphi_3, $$
where $F:[0,\delta_3)\rightarrow \mathbb{R}$ and $G:[0,\delta_3)\rightarrow \mathbb{R}$ are smooth functions such that $F=1$ near $r=0$ and $r=\delta_3$, and $G=\frac{r_3^2}{2}$ near $r=0$ and $r=\delta_3$. We want to choose $F$ and $G$ satisfying
$$F'(r_3)^2+G'(r_3)^2\neq 0,$$ 
to ensure that $d\tilde \eta$ is everywhere non-degenerate and hence that it defines a non-vanishing exact divergence-free vector field by duality with $\mu$. The helicity of the vector field $\tilde Y$ defined by the equation $\iota_{\tilde Y}\mu=d\tilde \eta$ is
\begin{align*}
    \mathcal{H}(\tilde Y)&= \int_M \tilde \eta \wedge d\tilde \eta \\
    &= \int_{M\setminus U_3} \eta\wedge d\eta + \int_{U_3} \tilde \eta \wedge d\tilde \eta.
\end{align*}
The first term does not depend on our choice of $F$ and $G$. The second term satisfies

\begin{align*}
     \int_{U_3} \tilde \eta \wedge d\tilde \eta&= \int_{U_3} \left(F(r_3)G'(r_3) - G(r_3)F'(r_3) \right)d\theta_3\wedge dr_3 \wedge d\varphi_3\\
     &=4\pi^2\left([FG]_0^{\delta_3}- \int_0^{\delta_3} 2G(r_3)F'(r_3) dr_3 \right).
\end{align*}
The first term is again independent of our choice of $F$ and $G$, so we simply want to prescribe the second term. First, fix some $F$ such that $F'>0$ in some interval $[a,b]\subset(0,1)$. We can choose $a$ such that $|G|<\delta$ in $[0,a]$ for any given $\delta$ (since $G=\frac{r_3^2}{2}$ near $r=0)$), and we can choose $b$ such that the size of $[b,1]$ is arbitrarily small. Then $\int_{I\setminus [a,b]} 2GF'dr_3$ can be assumed to be arbitrarily small in absolute value. We choose $G$ in $[a,b]$ arbitrarily, prescribing the value of $\int_{\tau_1}^{\tau_2} 2GF'$. Then the vector field $\tilde Y$ will have a prescribed value of the helicity. 

We have to check that $\tilde Y|_{U_3}= G'(r_3) \pp{}{\theta_3} - F'(r_3)\pp{}{\varphi_3}$ is homotopic to $Y_2|_{U_3}=\pp{}{\theta_3}$ relative to the boundary in $U_3$, so that the homotopy class of $\tilde Y$ is the same as that of $Y_2$ (and hence belongs to the prescribed initial homotopy class). We can construct an explicit homotopy in this case. For $t\in [0,1]$, consider a first homotopy:
$$ V_t=  \pp{}{\theta_3} - tF'(r_3)\pp{}{\varphi_3}, \enspace t\in [0,1].$$
Construct a second homotopy 
$$ \tilde V_s= ((1-s)+sG'(r_3))\pp{}{\theta_3}-F'(r_3)\pp{}{\varphi_3}, \enspace s\in [0,1]. $$
Both homotopies are homotopies of non-vanishing vector fields, since whenever $F'(r_3)$ vanishes we know that $G'(r_3)>0$. A concatenation of both homotopies gives a homotopy from $Y_2$ to $\tilde Y$ relative to the boundary, as we wanted.\\

To conclude, recall that we have used a different auxiliary volume form $\mu'$ for our computations of dual two-forms. The original volume form $\mu$ is equal to $H\mu'$ for a positive function $H\in C^\infty(M)$. Consider the vector field $\hat Y=\frac{1}{H} \tilde Y$. The orbits $\gamma_i$ are still non-degenerate periodic orbits of $\hat Y$, and since $\iota_{\hat Y}\mu= \iota_{\tilde Y} \mu'=d\tilde \eta$, the property $\iota_{\gamma_1} \tilde \eta>0 $ and $\int_{\gamma_2}\tilde \eta<0$ is satisfied by a primitive of $\iota_{\hat Y}\mu$. Furthermore, the helicity of $\hat Y$ with respect to $\mu$ is the same as that of $\tilde Y$ with respect to $\mu'$, which hence can be prescribed to take any value. 
\end{proof}
Our next goal will be to construct vector fields whose vorticity is in the adjoint orbit (under the action of $\operatorname{Diff}_\mu^0(M)$) of some vector fields constructed via Lemma \ref{lem:Hreeb} and Proposition \ref{prop:Hnonreeb}, but which have prescribed large enough values of the energy.

\subsection{Local vortex stretching}
In this subsection, we start with an arbitrary exact divergence-free vector field $X$ (not identically zero) and prescribe any large enough value of the energy of a vector field $u$ whose vorticity is in the adjoint orbit of $X$.

\begin{theorem}[Local vortex stretching]\label{thm:localvortex}
Let $X$ be an exact divergence-free vector field on a closed Riemannian three-manifold that does not vanish everywhere. Then for each large enough positive real number $e$, there exists a divergence-free vector field $u_e$ such that:
\begin{itemize}
    \item[-] $\operatorname{curl}(u_e)=Y$, where $Y=\varphi_*X$ for some volume-preserving diffeomorphism $\varphi$ of $M$ isotopic to the identity, compactly supported near any point $p$ such that $X|_p\neq 0$
    \item[-] $\mathcal{E}(u_e)=e$.
\end{itemize}
\end{theorem}

\begin{proof}
For a fixed vector field $X$ not everywhere vanishing, consider the set
$$ \mathcal{A}_0(X)=\{(\varphi)_*(X)  \mid \varphi \in \operatorname{Diff}_\mu^0(M) \text{ and }t \in \mathbb{R}\},  $$
where $\operatorname{Diff}_\mu^0(M)$ is the set of volume-preserving diffeomorphisms that are isotopic to the identity. The fact that the helicity is invariant under volume-preserving transformations implies that each $Y\in \mathcal{A}_0(X)$ satisfies $\mathcal{H}(Y)=\mathcal{H}(X)$. We define the set 
$$ E=\{\mathcal{E}(\curl^{-1}Y)  \mid  Y\in \mathcal{A}_0(X)\}.$$
The following lemma reduces the proof of Theorem \ref{thm:localvortex} to showing that $E$ is unbounded by means of composing $X$ with a compactly supported volume-preserving isotopy.

\begin{lemma}\label{lem:A}
If $E$ is not bounded, then for each $e\geq \mathcal{E}(\curl^{-1}X)$, there exists $Y\in \mathcal{A}_0(X)$ such that $\mathcal{E}(\curl^{-1}Y)=e$. Furthermore, if we can find vector fields in $E$ with arbitrarily large energy of the form $F_*X$ with $F$ a volume-preserving and compactly supported near a point $p$, then we can choose $Y=G_*X$ with $G$ compactly supported near $p$ too.
\end{lemma}

\begin{proof}[Proof of Lemma \ref{lem:A}]
    If $E$ is not bounded then it is necessarily not bounded from above, so given a real number $e>0$ we can find some other number $M>e$ and a vector field $Y\in \mathcal{A}_0(X)$ such that $\mathcal{E}(\curl^{-1}(Y))=M$. By definition of $\mathcal{A}_0(X)$, the vector field $Y$ is given by $\varphi_*X$ where $\varphi$ is a volume-preserving diffeomorphism of $M$ which is isotopic to the identity. Let $\varphi_t$ be an isotopy such that
    \begin{itemize}
        \item[-] $\varphi_t=\varphi$ for $t\in [1-\delta,1]$,
        \item[-] $\varphi_t=\operatorname{id}$ for $t\in [0,\delta]$,
    \end{itemize} 
    where $\delta>0$ is a small enough positive number. We will now apply Moser's path method to show that $\varphi_1$ is isotopic to the identity through volume-preserving diffeomorphisms. For each $t\in [0,1]$, we write $\mu_t=\varphi_t^*\mu$. By construction, it is satisfied that for
     $t\in [0,\delta]\cup [1-\delta,1]$ we have $\mu_t=\mu$. Consider the two-parametric family of volume forms 
     $$\mu_{t,s}=(1-s)\mu_t+s\mu.$$
     This family satisfies
      \begin{equation}\label{eq:boundary}
     \mu_{t,s}=\mu \quad t\in [0,\delta]\cup [1-\delta,1].
     \end{equation} 
    Note as well that
    $$\int_M \mu_t= \int_M \varphi_t^*\mu= \int_M\mu,$$ which implies that $[\mu_{t,s}]=[\mu]$ in cohomology. It follows that $\mu_{t,s}-\mu=d\eta_{t,s}$ for some family of one-forms $\eta_{t,s}$. Equation \eqref{eq:boundary} implies that we can assume that $\eta_{t,s}=0$ for $t\in [0,\delta]\cup [1-\delta,1]$. Using that $\mu_{s,t}$ is a volume form for each value of the parameters, for a fixed $t$ the equation
    \begin{equation}\label{eq:mos}
        \mathcal{L}_{X_{s,t}}\mu_{s,t}=-d\eta_{s,t},
    \end{equation}  
    admits a unique solution $X_{s,t}$, that we understand as time-dependent vector fields in $M$, where $s$ represents the time-parameter and $t$ parametrizes the family of vector fields. Indeed, we can find a solution to Equation \eqref{eq:mos} by solving instead
    $$\iota_{X_{s,t}}\mu_{s,t}=-\eta_{s,t},$$ which admits a unique solution because $\mu_{s,t}$ is a volume form for each value of $s$ and $t$. The family $X_{s,t}$ integrates to a $t$-parametric family of isotopies $$\psi_{s}^t:M\longrightarrow M,$$
    that satisfy $(\phi_1^t)^*\mu_t=\mu$ for each $t\in [0,1]$. Furthermore, this isotopy is constant for $t$ near $[0,\delta]\cup [1-\delta,1]$. The family of diffeomorphisms $G_t= \psi_1^t \circ \varphi_t$ satisfies
    \begin{itemize}
        \item [-] $G_0=\operatorname{id}$,
        \item [-] $G_1=\varphi_1$,
        \item [-] ${G_t}^*\mu= (\psi_1^t)^*(\varphi_t^*\mu)=(\psi_1^t)^*\mu_t=\mu$.
    \end{itemize}
    We have thus proved that $\varphi_1$ is isotopic to the identity via volume-preserving diffeomorphisms.

    To conclude, we use that the energy functional $\mathcal{E}$ is continuous with the $C^0$-topology, and so is the inverse curl operator by Lemma \ref{lem:curlinverse}. By hypothesis, we have
    $$\mathcal{E}(\curl^{-1}Y)=\mathcal{E}(\curl^{-1}({G_1}^*X))=M> \mathcal{E}(\curl^{-1}X)=\mathcal{E}(\curl^{-1}({G_0}^*X)),$$ thus, by continuity, we deduce that there exists some $\tilde t\in [0,1]$ such that 
    $$\mathcal{E}(\curl^{-1}({G_{\tilde t}}^*(X)))=e.$$ By construction $G_{\tilde t}$ is volume-preserving and isotopic to the identity, so it follows that ${G_{\tilde t}}^*(X)$ is in $\mathcal{A}_0(X)$. Observe that since $\mu_{t,s}=\mu$ away from the support of $\varphi_t$, it follows that the support of $G_{\tilde t}$ is contained in the support of $\varphi_t$. This shows that if $\varphi_t$ is compactly supported near a point, then so is $G:=G_{\tilde t}$.
\end{proof}

It is left to show that $E$ is not bounded from above by considering the pushforward of $X$ by compactly supported volume-preserving isotopies near an arbitrary point $p$ where $X|_p\neq 0$. We can find a flow-box neighborhood of $X$ near $p$ of the form $U=D^2\times [0,1]$, with flow-box coordinates $(x,y,z)$, where the vector field takes the form $X=\pp{}{z}$. The fact that $X$ preserves $\mu$ implies that the volume form has the form $\mu=f(x,y)dx\wedge dy \wedge dz$ in these coordinates. Up to a change of coordinates in the $D^2$ factor, we can assume that $\mu=dx\wedge dy \wedge dz$.

Let $(r,\theta)$ be polar coordinates of $D^2$, and for each $K>0$ consider the diffeomorphism
\begin{align*}
    F_K:D^2\times I &\longrightarrow D^2\times I\\
    (r,\theta,z) &\longmapsto (r,\theta + K\varphi(r)\psi(z),z),
\end{align*}
where $\varphi$ and $\psi$ are smooth functions satisfying the following conditions. The function 
$$\varphi:[0,1]\longrightarrow \mathbb{R}$$
is non-negative, equal to $0$ near $r=0$ and $r=1$, and equal to $1$ for $r\in [1/3,2/3]$. The function 
$$\psi:[-1,1]\longrightarrow \mathbb{R}$$ is a fixed bump function that is positive somewhere in $(-1,1)$. The diffeomorphism $F_K$ is volume-preserving and compactly supported in $U$, hence it extends to $M$ as the identity away from $U$. Computing the pushforward of $X$ by $F_K$ we obtain
$$Y=(F_K)_*(X)=K\varphi(r)\psi'(z)\pp{}{\theta}+\pp{}{z}, $$
a representation of how the integral curves of $X$ change under $F_K$ in a flow-box is shown in Figure \ref{fig:vortex}.

   \begin{figure}[!h]
  \begin{center}
 \begin{tikzpicture}
      \node[anchor=south west,inner sep=0] at (0,0) {\includegraphics[scale=0.08]{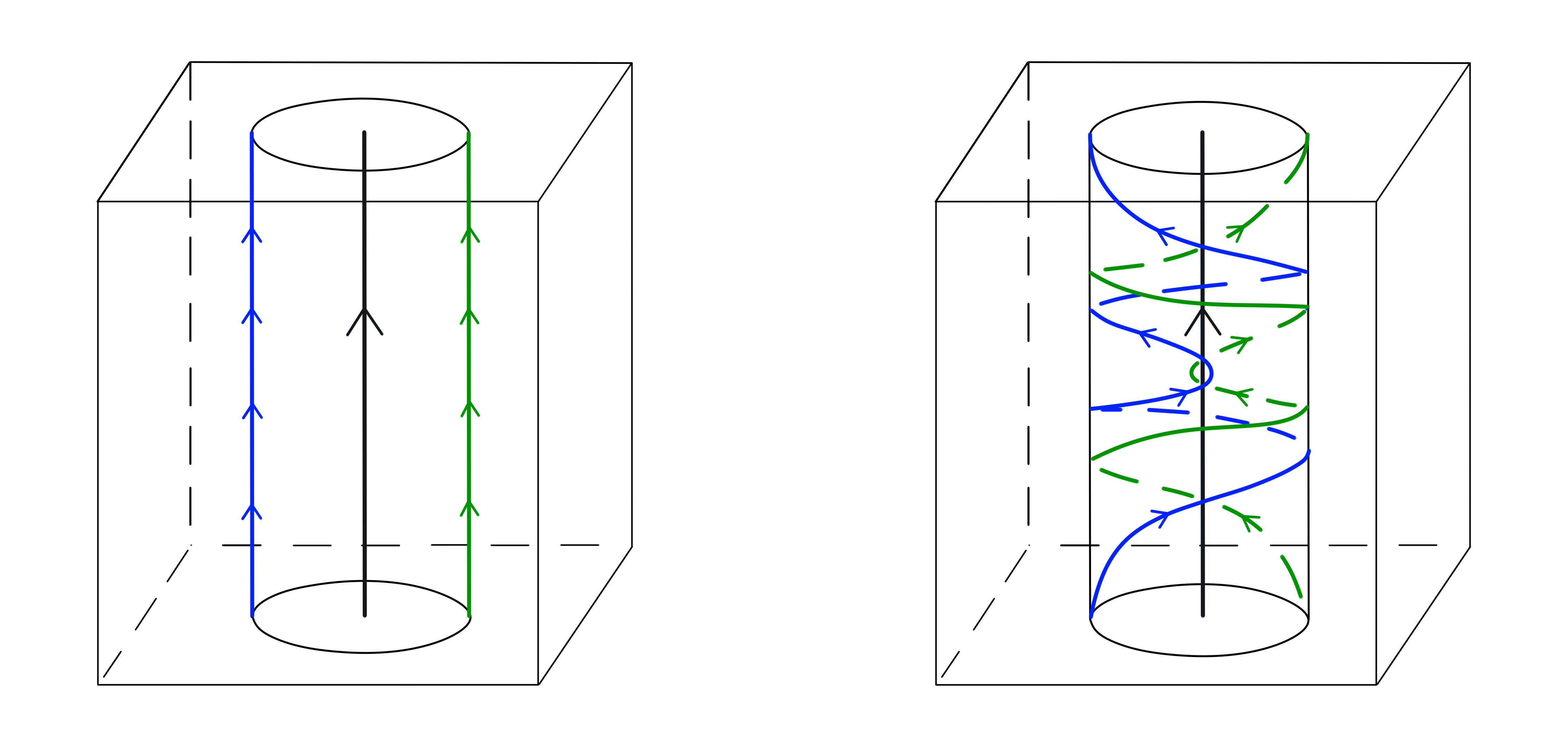}};
      \draw[->] (5.6,3)--(7,3);
      \node at (6.3, 3.5) {${F_K}_*$};
 \end{tikzpicture}
 \centering
 \caption{Flowlines of $X$ and of $(F_K)_*(X)$.}
  \label{fig:vortex}
 \end{center}
 \end{figure}

The vector field $Y$ splits as $KZ+X$, where $Z$ is compactly supported in $U$ and $K>0$ is an arbitrarily large constant. The energy of $\curl^{-1}Y=K\curl^{-1} Z+\curl^{-1} X$ is
\begin{equation}\label{eq:split}
    \mathcal{E}(\curl^{-1} Y)=\mathcal{E}(\curl^{-1} X)+K^{2} \mathcal{E}(\curl^{-1} Z)+2K\int_{U} g(\curl^{-1} X, \curl^{-1} Z).
\end{equation}
Let $C_{1}, C_{2}$ be constants verifying
\[
C_{1} \geq \bigg| 2 \int_{U} g(\curl^{-1} X, \curl^{-1} Z)\bigg|, \,\,\, C_{2} \leq  \mathcal{E}(\curl^{-1} Z),
\]
then we have
\[
\mathcal{E}(\curl^{-1} Y) \geq \mathcal{E}(\curl^{-1} X)+K^{2}\bigg(C_{2}-\frac{C_{1}}{K}\bigg).
\]
Therefore, provided $\curl^{-1} Z$ is not identically zero and we can choose $C_2>0$, for $K$ large enough equation $C_{2}-\frac{C_{1}}{K}>\frac{1}{2}C_{2}$ implies
\[
\mathcal{E}(\curl^{-1} Y) \geq  \mathcal{E}(\curl^{-1} X)+\frac{1}{2} K C_{2}
\]
and the lemma follows. It is clear that by construction $Z$ is a divergence-free vector field that is necessarily exact: its support is a simply connected domain. Since it is not everywhere vanishing, the vector field $\curl^{-1}Z$ is not identically zero.
\end{proof}

We have called Theorem \ref{thm:localvortex} ``local vortex stretching" since not only the energy of the inverse curl of the vector field is increased by this modification, but also the length of the integral curves of the vorticity field, as depicted in Figure \ref{fig:vortex}, and its energy.

\section{Applications to ideal hydrodynamics}\label{s:main}
In this section, we will use our previous constructions to prove the non $C^1$-mixing of the Euler equations. We give an additional application of our discussion to stationary solutions in the round sphere.

\subsection{Non $C^1$-mixing of the Euler equations}

On a Riemannian closed three-manifold, we denote by $$\varphi_t:\mathfrak{X}_\mu(M)\longrightarrow \mathfrak{X}_\mu(M),$$
the local flow defined by the Euler equation on the space of smooth divergence-free vector fields. That is, for an initial condition $u_0$, the time-dependent vector field $u_t=\varphi_t(u_0)$ solves the Euler equation with initial condition $u_0$. The flow $\varphi_t$ is locally defined near $t=0$ for any initial condition, see \cite{EM}, but it is not known whether it is defined for each $t\in \mathbb{R}$.

For a given homotopy class of non-vanishing vector fields $a$, denote by $\mathcal{V}_{a,h,e}$ the set of volume preserving flows in $M$ with energy $e$, and whose curl is everywhere non-vanishing, in the homotopy class $a$ and with helicity $h$. As discussed in Section \ref{ss:eul}, this set is invariant by the Euler equations.

\begin{theorem}\label{thm:nonmix}
    Let $(M,g)$ be a closed Riemannian three-manifold. Fix a homotopy class $a$, a non-zero value of the helicity $h\neq 0$, and an energy $e \geq e_{0}$ (where $e_{0}$ depends on $h$ and $a$). Then there exist two $C^1$-open sets $C_{a,h,e}\subset \mathcal{V}_{a,h,e}$ and $N_{a,h,e}\subset \mathcal{V}_{a,h,e}$ such that
    $$ \varphi_t(C_{a,h,e})\cap N_{a,h,e} =\emptyset, $$
    for each $t$ where the flow is defined.
\end{theorem}

\begin{proof}
For fixed homotopy class of vector fields $a$, real number $h\neq 0$ consider the set
 $$\mathcal{A}_{a,h}=\{X \in \mathfrak{X}_\mu^0(M)\mid 
 [X]=a, \mathcal{H}(X)=h\}.$$
By Lemma \ref{lem:Hreeb}, there is a contact type vector field $Y_0 \in \mathcal{A}_{a,h}$. By Theorem \ref{thm:localvortex}, for any large enough value of $e$ we can find a volume-preserving vector field $u_0$ such that $\curl(u_0)=F^*Y_{0}$ (we write it as a pull-back instead of a pushforward to simplify a bit the notation) for some volume-preserving diffeomorphism $F:M \rightarrow M$ isotopic to the identity. Thus, as was explained in subsection \ref{ss:contacttype}, $\curl(u_0)$ is of contact type as well and $u_t$ is a contact type solution according to Definition \ref{def:contactsolution}.

%To see this, since $Y$ is of contact type we know that $\iota_Y\mu=d\alpha$ for a contact form $\alpha$. Then, we have
%\begin{align}\label{eq:contactinvariance}
%\begin{split}
   % \iota_{\operatorname{curl} u_0}\mu&=\iota_{F^*Y}\mu\\
   % &=F^*(\iota_{Y}\mu)\\
    %&=dF^*\alpha 
    %\end{split}
%\end{align}and since $\alpha$ is a contact form, so is $F^*\alpha$. We deduce that $\curl (u_0)$ is of contact type. 

By Theorem \ref{thm:C0obs}, there is a $C^0$-neighborhood of $\curl (u_0)$ in $\mathfrak{X}_\mu^0(M)$ such that every vector field in such neighborhood is of contact type. By the continuity properties of the curl operator (Lemma \ref{lem:curl}) we conclude that there exist a $C^1$-neighborhood $\mathcal{B}$ of $u_0$ in $\mathfrak{X}_\mu(M)$ such that every $u\in \mathcal{B}$ satisfies that $\curl(u)$ is of contact type. We define the open set $C_{a,h,e}$ of $\mathcal{V}_{a,h,e}$ as
$$C_{a,h,e}= \mathcal{B}\cap V_{a,h,e}. $$

On the other hand, by Proposition \ref{prop:Hnonreeb} there exists $Y_1 \in \mathcal{A}_{a,h}$ satisfying the hypotheses of Theorem \ref{thm:C0obs}: the two form $\iota_{Y_1}\mu$ admits a primitive $\beta$ that integrates positively on a non-degenerate periodic orbit $\gamma_{1}$, and negatively on another non-degenerate periodic orbit $\gamma_{2}$. By Theorem \ref{thm:localvortex}, for any large enough value of $e$ we can find a volume-preserving vector field $u_1$ such that $\curl u_1=G^*Y_{1}$, for some volume-preserving diffeomorphism $G:M\rightarrow M$ isotopic to the identity. Let us show that $G^*Y_{1}$ also satisfies the hypotheses of Theorem \ref{thm:C0obs}. Arguing as in \eqref{eq:contactinvariance}, we have $\iota_{G^*Y_{1}}\mu=dG^{*}\beta$. The orbits $\tilde \gamma_1=G^{-1}(\gamma_1)$ and $\tilde \gamma_2=G^{-1}(\gamma_2)$ are non-degenerate periodic orbits of $G^*Y_{1}$ and we have
%Since $\iota_{Y_{1}}\mu=d\beta$, we know that $Y_{1}$ has two non-degenerate periodic orbits $\gamma_1,\gamma_2$ such that $\int_{\gamma_1}\beta>0$ and $\int_{\gamma_2}\beta<0$. 
\begin{align*}
    \int_{\tilde \gamma_1} G^*\beta&= \int_{G(\tilde \gamma_1)}\beta=\int_{\gamma_1} \beta>0,\\
    \int_{\tilde \gamma_2} G^*\beta&= \int_{G(\tilde \gamma_2)}\beta=\int_{\gamma_2} \beta<0.
\end{align*}
We conclude that $G^*Y_{1}$ also satisfies the hypotheses of Theorem \ref{thm:C0obs}, and we deduce that there is $C^0$-neighborhood of $\curl u_1=G^*Y_{1}$ in $\mathfrak{X}_\mu^0(M)$ such that every vector field in this neighborhood is not of contact type. By the continuity properties of the curl operator, see Lemma \ref{lem:curl}, we conclude that there is a $C^1$-open neighborhood $N$ of $u_1$ in $\mathfrak{X}_\mu(M)$ such that every $u\in C$ satisfies that $\curl(u)$ is not of contact type. We define the open set $N_{a,h,e}$ of $\mathcal{V}_{a,h,e}$ as
$$N_{a,h,e}= C\cap V_{a,h,e}. $$

 We claim that $\varphi_t(C_{a,h,e}) \cap N_{a,h,e}=\emptyset$ for each value of $t$ where the flow of the Euler equation is defined. First, we have seen that the contact type (and the non contact type) property of an exact divergence-free vector field is preserved under volume-preserving transformations. The transport of vorticity \eqref{eq:Helm} tells us that if we look at the evolution of the Euler equations with initial condition $u_0$, then the vorticity at time $t$ is always diffeomorphic to the vorticity of $u_0$. Hence if there was some $v\in C_{a,h,e}$ and some $t$ such that $\varphi_t(v)\in N_{a,h,e}$, the vorticity of $\varphi_t(v)$ is simultaneously of contact type and not of contact type, thus reaching a contradiction.
 \end{proof}

 \begin{Remark}\label{rem:regularity}
While we have assumed our differential forms and vector fields to be $C^{\infty}$ throughout the proof, it is not hard to check that the same arguments provide, mutatis mutandis, an analog of Theorem \ref{thm:nonmix} for non-mixing in the $C^{1}$-topology in the space of $C^{k, s}$ velocity fields (and thus for $C^{k-1, s}$ vorticities), as long as $k\geq 1$ and $s>0$. Indeed, for example for $k=1$, the curl operator is a continuous and surjective map between $C^{1,s}$ vector fields and $C^{0,s}$ exact div-free vector fields (where the divergence is understood in the weak sense when $k=1$, or in the usual sense for $k>1$). By virtue of Lemma \ref{lem:C1bound}, it makes sense to define contact type vorticities of $C^{0, s}$-regularity as those with a $C^{1,s}$ primitive contact 1-form. Having clarified this, the discussion in Sections \ref{ss:ctvorticities} and \ref{ss:obs} applies straightforwardly (notice we only use one of the implications in McDuff's characterization, which works for low regularity two-forms), as do the constructions in Section \ref{s:data}, and then the proof in this section. Observe that $C^{1, s}$ is the lowest regularity for which local existence and uniqueness of the Euler flow $\varphi_t$ holds (and non-mixing results are of interest independently of the existence of finite-time singular solutions, which were recently proven to exist in $\R^{3}$ for this regularity \cite{Elgindi}). This improves the regularity for which non-mixing can be established, since in \cite{KKPS} it was needed that velocities are at least of $C^{4,s}$-regularity.
\end{Remark}

%The proof of Theorem \ref{thm:nonmix} relies on the fact that the contact type property (or the lack of it) of the curl of the velocity field of the fluid is invariant by the Euler equations. Hence, given a time-dependent solution $u_t$ to the Euler equations whose vorticity is non-vanishing, one can say without ambiguity that the solution $u_t$ is (or not) of contact type.

%\begin{defi}\label{def:contactsolution}
    %A solution $u_t$ to the Euler equations in $(M,g)$ whose curl is non-vanishing is called of \textit{contact type} if its curl is of contact type.
%\end{defi}

%Analogously, we say that an initial condition $u_0\in \mathfrak{X}_\mu(M)$ is of contact type if the curl of $u_0$ is of contact type.\\

The proof of Theorem \ref{thm:nonmix} relies on the fact that the contact type property (or the lack of it) of the curl of the velocity field of the fluid is invariant by the Euler equations. We will now see that the isotopy class of the contact structure associated with a contact type vorticity field is also an invariant of $u_t$. In other words, the contact structure is ``transported", since it is inherent to the vorticity. Thus, when several isotopy classes of contact structures exist in a given homotopy class of plane fields, these produce several $C^1$-open sets in $\mathcal{V}_{a,h,e}$ that do not intersect when we look at their evolution by the Euler equations.

\begin{lemma}\label{lem:continv}
    Let $u_t$ be a contact type solution to the Euler equations. Write $\iota_{\curl u_0}\mu=d\alpha$ for a contact form $\alpha$. Then the isotopy class of the contact structure $\xi=\ker \alpha$ does not depend on the chosen one-form and is an invariant of $u_t$.
\end{lemma}
\begin{proof}
    The fact that the isotopy class of $\ker \alpha$ does not depend on the primitive is shown in Lemma \ref{lem:Xcont}. Transport of vorticity, in the formulation \eqref{eq:transport}, shows that
        $$\curl u_t= (\phi_t)_*(\curl u_0),$$
        or, equivalently, that if we write $\iota_{\curl u_t}=d\alpha_t$ with $\alpha_t=g(u_t,\cdot)$, we have
        $$ d\alpha_t=(\phi_t)^*(d\alpha_0). $$
        The family of one forms $\tilde \alpha_t=(\phi_t)^*\alpha_0$ satisfy $d\tilde \alpha_t=d\alpha_t$, and $\tilde \alpha_t$ is a contact form for each $t$. We deduce that $\ker \tilde \alpha_t$ is contact isotopic to $\ker \alpha_0$.
\end{proof}

Given $\xi$, we denote by $\mathcal{V}_{[\xi]}\subset \mathcal{V}_{a,h,e}$ (we do not specify $a,h,e$ to avoid overloading the notation) the $C^1$-open set of initial conditions whose curl is of contact type with a contact form defining a contact structure isotopic to $\xi$. Because of the invariance under the Euler equations, one might wonder whether Theorem \ref{thm:nonmix} wouldn't readily follow by setting $C_{a,h,e}=\mathcal{V}_{[\xi]}$ and $N_{a, h, e}=\mathcal{V}_{[\xi']}$ for $[\xi] \neq [\xi']$. The number of isotopy classes of contact structures on a given homotopy class $a$ of plane fields might vary wildly from one manifold to another, or from one homotopy class to another. In dimension three, though, it follows from a combination of strong results in contact topology \cite{El2,CGH} that all but finitely many homotopy classes of plane fields admit a single isotopy class of contact structures, so Lemma \ref{lem:continv} is clearly insufficient to deduce Theorem \ref{thm:nonmix}. In $S^3$ every homotopy class of plane fields has exactly one isotopy class of contact structures except the trivial homotopy class (the one given by the plane field orthogonal to the Hopf fibration), which has exactly two isotopy classes \cite{El}. On the other hand, in the three-torus there exists a particular homotopy class of plane fields $b$ that admits infinitely many isotopy classes of contact structures \cite{H2}. These define countably many invariant $C^1$ open sets in $\mathcal{V}_{a,h,e}$, where $a$ is fixed to be the homotopy class of vector fields in correspondence with $b$. 

Nevertheless, contact topology provides us with invariants way finer than simply the isotopy class of a contact structure, and these can potentially be used to obtain a better picture of the phase space of the Euler equations in $\mathcal{V}_{[\xi]}$. This is what we will explain in Section \ref{s:ech}.

\subsection{Stationary solutions in the three-sphere}

The KKPS family of stationary solutions to the Euler equations in the round sphere $(S^3,g_{std})$ was introduced in \cite{KKPS} and further studied in \cite{Sl}, where it was shown that the family is not isolated (in the sense that there are steady solutions not in the family but arbitrarily close to its elements in the $C^{k}$-norm). To describe the family, we look at $S^3$ as the set of unit vectors in $\R^4$ with coordinates $(x,y,z,\xi)$. Consider the Hopf and the anti-Hopf fields
\begin{align*}
    X_1&=-y\pp{}{x}+x\pp{}{y}+\xi\pp{}{z}-z\pp{}{\xi},\\
    X_2&=-y\pp{}{x}+x\pp{}{y}-\xi\pp{}{z}+z\pp{}{\xi}.
\end{align*}
Define the function $F=(x^2+y^2)|_{S^3}$. The KKPS family is given by vector fields of the form
$$ X=f(F) X_1+g(F)X_2,$$
where $f,g$ are smooth functions, is always a stationary solution to the Euler equations. Its Bernoulli function is given by $B=\int_0^FH(s)ds$ where $H(F)=ff'+gg'+4fg+(2F-1)(fg'+gf')$. 
\begin{lemma}\label{lem:S3ex}
    There exists a KKPS stationary solution to the Euler equations in $(S^3,g_{std})$ that satisfies the hypotheses of Theorem \ref{thm:C0obs}.
\end{lemma}
\begin{proof}
    Considering the complex numbers $(z_1,z_2)=(x+iy,z+i\xi)$ and the coordinates of $S^3$ given by $(z_1,z_2)|_{S^3}=(\cos s e^{i\phi_1},\sin se^{i\phi_2})$, with $s\in [0,\pi/2]$ and $\phi_1,\phi_2\in [0,2\pi]$, we have
$$ X_1=\pp{}{\phi_1}+\pp{}{\phi_2}, \quad X_2=\pp{}{\phi_1}-\pp{}{\phi_2}.  $$
The general KKPS solution $X$ is of the form
\begin{align*}
    X&=f(\cos^2s)X_1+ g(\cos^2s)X_2 \\
     &= \left(f(\cos^2s)+g(\cos^2)\right)\pp{}{\phi_1} + \left(f(\cos^2s)-g(\cos^2)\right)\pp{}{\phi_2}.
\end{align*}
We choose $f$ such that $f=2\cos^2s$ near $s=0$ and $f=2(1-\cos^2s)$ near $s=\pi/2$. Choose $g=g_1+g_2$ where 
\begin{enumerate}
    \item $g_1=2\tau \cos^2s$ near $s=0$, where $\tau$ is some irrational number,
    \item $g_1=2\tau(1-\cos^2s)$ near $s=\pi/2$,
    \item $g_2$ is such that $g_2=0$ near $s=0$ and equal to $1$ near $s=\pi/2$, and $\int_0^sg_2=C$ for some constant $C>0$.
\end{enumerate}
We can easily assume $f^2+g^2\neq 0$ everywhere. Observe that $\{s=0\}$ and $\{s=\pi/2\}$ define non-degenerate elliptic periodic orbits of $X$ (since we chose $\tau$ irrational), since $X|_{s=0}=2(1+\tau)\pp{}{\phi_1}$ and $X|_{s=\pi/2}=\pp{}{\phi_2}$.

Consider the one-form
$$ \alpha=\left(\frac{1}{2}\int_0^sg-fds - K\right) d\phi_1 +\left( \frac{1}{2}\int_0^s f+gds\right)\phi_2. $$

First, choosing $C$ large enough we can assume that $\frac{1}{2}\int_0^s f+gds>0$ near $s=\pi/2$. Secondly, once we have chosen $C$, we can choose $K>0$ such that $\frac{1}{2}\int_0^sg-fds - K=0$ at $s=\pi/2$. This ensures that $\alpha$ is well defined in $S^3$, since it the coefficient of $d\phi_1$ vanishes quadratically in $s$ near $s=\pi/2$ and the coefficient of $d\phi_2$ vanishes quadratically in $s$ near $s=0$. By construction, we have $\alpha(X)>0$ near $\{s=\pi/2\}$ and $\alpha(X)<0$ near $\{s=0\}$. It is straightforward to check that $\iota_X\mu=d\alpha$, where $\mu=-2\cos s \sin sd\phi_1d\phi_2ds$ is the induced Riemannian volume. We conclude that $\int_{\{s=0\}}\alpha<0$ and $\int_{\{s=\pi/2\}}\alpha>0$, where we have oriented the circles with the positive orientation of $X$. Since the periodic orbits are non-degenerate, we are in the hypotheses of Theorem \ref{thm:C0obs}.
\end{proof}

As a corollary, we deduce the non $C^0$-density of rotational Beltrami fields among stationary solutions to the Euler equations. Notice that, in particular, curl eigenfields (also known as strong Beltrami fields) are not $C^0$-dense in the set of stationary solutions to the Euler equations neither, since in this case $\curl X=\lambda X$ with $\lambda\neq 0$.

\begin{corollary}\label{cor:statdens}
Consider $S^3$ equipped with the standard round metric. There exist non-vanishing stationary solutions to the Euler equations whose $C^0$-neighborhood does not contain any rotational Beltrami field.
\end{corollary}

\begin{proof}
In $S^3$, any non-vanishing rotational Beltrami field $Y$ satisfies $\iota_Y\mu=f d\alpha$, where $f$ is everywhere non-vanishing and $\alpha=Y^\flat$. In particular, $\iota_Y\mu$ is always of contact type, see also Remark \ref{rem:rotBelt}. By Lemma \ref{lem:S3ex} and Theorem \ref{thm:C0obs}, there exists a stationary solution to the Euler equations whose $C^0$-neighborhood does not contain divergence-free vector fields of contact type. 
\end{proof}

\begin{corollary}\label{cor:KKPSint}
    There exists a KKPS stationary solution such that every stationary solution to the Euler equations in a $C^0$-neighborhood necessarily admits a non-trivial smooth first integral.
\end{corollary}

\begin{proof}
    A stationary solution in the $C^0$-neighborhood of the solution constructed in Lemma \ref{lem:S3ex} is not of contact type by Theorem \ref{thm:C0obs}. Let $Y$ be a stationary solution in that neighborhood and assume it admits no non-trivial smooth first integral. We know that $\iota_Yd\alpha$ is exact, where $\alpha=Y^\flat$. If $\iota_Yd\alpha=dB$ with $B$ non-constant, then $B$ is a non-trivial first integral of $Y$.  Otherwise, the field $Y$ is a Beltrami field and $\curl Y=fY$ for some $f\in C^\infty(S^3)$. The proportionality factor $f$ is always a first integral of $Y$, so $f$ has to be constant, say equal to some $c\in \mathbb{R}$. This constant cannot be zero, because zero is not an eigenvalue of the curl operator, and hence $\alpha\wedge d\alpha=c\mu \neq 0$. This shows that $Y$ is a contact type (stationary) solution to the Euler equations, since its vorticity $\curl Y$ is such that $\iota_{\curl Y}\mu=d\alpha$ for a contact type two-form $d\alpha$, and we reach a contradiction. So $Y$ always admits a (non-trivial) smooth first integral.
\end{proof}

\begin{Remark}
The proof of Corollary \ref{cor:statdens} does not make use of the ambient geometry, except for finding a particular stationary solution. A natural question is if there is a simple proof that this particular solution cannot be $C^0$-approximated by rotational Beltrami fields that exploits the fact that the ambient metric is the round one.
\end{Remark}

\section{Contact invariants in the Euler equations}\label{s:ech}
 
We have shown in Section \ref{s:main} that any isotopy class of contact structures $[\xi]$ in a three-manifold defines a $C^1$-open set $\mathcal{V}_{[\xi]}\subset \mathcal{V}_{a,h,e}$ that is invariant by the Euler equations, namely the set of all those volume-preserving vector fields in $\mathcal{V}_{a,h,e}$ whose curl lies in $\mathcal{C}_\mu^+ (M,\xi)$ (or $\mathcal{C}_\mu^-(M,\xi)$ if $h<0$). In this section, we deepen the connection between contact geometry and hydrodynamics by associating invariants coming from the field of contact topology and Seiberg-Witten theory to contact type solutions of the Euler equations. In this setting the natural topology in $\mathcal{V}_{a,h,e}$ will be the topology induced by the $C^{1,s}$-topology in $\mathfrak{X}_\mu(M)$, for which the set $\mathcal{V}_{[\xi]}$ is also an open set (invariant by the Euler equations).

\subsection{Embedded contact homology spectral invariants}\label{ss:ECH}

Given a contact structure $\xi$ and a defining contact form $\alpha$ on a three-dimensional manifold $M$, the action spectrum $\mathcal{A}(M,\alpha)$ of $\alpha$ is the set of (possibly non-minimal) periods of the closed orbits of the Reeb field $R$ of $\alpha$, i.e.
$$\mathcal{A}(M,\alpha)=\bigg\{\int_\gamma \alpha  \enspace | \enspace \gamma \in C^\infty(S^1, M), \dot \gamma=R(\gamma) \bigg\}  \subset \mathbb{R}_+.$$
We will now briefly introduce embedded contact homology and the ECH spectral invariants, we refer to \cite{Hut2, Hut} for more details. Let us recall that a periodic orbit of a vector field is non-degenerate if the eigenvalues of the linearized Poincar\'e map of the periodic orbit are not roots of the unity. Otherwise, we say that the periodic orbit is degenerate. Momentarily, we will assume that the contact form $\alpha$ is non-degenerate, meaning that all the periodic orbits of the Reeb field $R$ defined by $\alpha$ are non-degenerate. For any element $\Gamma \in H_1(M;\mathbb{Z})$, one can define a homology of a chain complex generated over $\mathbb{Z}/2$ by sums of pairs $(\gamma_i,m_i)$, where $\gamma_i$ are distinct embedded closed orbits of $R$ and $m_i$ are positive integers ($m_i=1$ if $\gamma_i$ is hyperbolic) and whose sum lies in the homology class $\Gamma$. This is called ``embedded contact homology" and is denoted by $\ech(M,\alpha,\Gamma)$. This homology depends only on $\xi$ and hence is written as well as $\ech(M,\xi,\Gamma)$, so that for any other choice of non-degenerate contact form $\alpha_1=f\alpha$ for some positive function $f\in C^\infty(M)$, there is a natural isomorphism 
\begin{equation}\label{eq:ECHiso1}
    \ech(M,\alpha,\Gamma)\cong \ech(M,\alpha_1,\Gamma)\cong \ech(M,\xi,\Gamma). 
\end{equation}
For any $L>0$, there exists also a  well-defined filtered embedded contact homology $\ech^L(M,\alpha,\Gamma)$, where one considers only the homology of the subcomplex generated by orbits whose period is strictly less than $L$. There is a natural inclusion $$\iota_L : \ech^L(M,\alpha,\Gamma) \longrightarrow \ech (M,\alpha,\Gamma).$$ Out of this inclusion, one can extract different spectral invariants of $\alpha$, such as spectral invariants defined by non-vanishing ECH classes, or the ``full ECH spectrum" \cite[Section 3]{Hut}. We will use the latter for technical reasons.
\begin{defi}\label{def:spectral}
    For each positive integer $k>0$, we define $\widetilde c_k(M,\alpha)$ to be the infimum over all $L\in \mathbb{R}$ such that the image of $\iota_L: \ech^L(M,\alpha, 0) \rightarrow \ech (M,\alpha,0) $ has dimension at least $k$.
\end{defi}
These numbers extend continuously to any (possibly degenerate) contact form, i.e. contact forms whose Reeb field possibly admits degenerate periodic orbits. The full ECH spectrum can be defined as well for non-vanishing $\Gamma$, but we stick to $\Gamma=0$ because this hypothesis will be needed anyway for our purposes (see Lemma \ref{lem:spectwoform} below).

 The spectral invariants satisfy several properties, we mention the ones that are relevant to this work.
\begin{itemize}
\vspace{0.2cm}
    \item[-] ($C^0$-Continuity) For any sequence $f_j\in C^\infty(M, \mathbb{R}_+)$ such that $\norm{f_j-f}_{C^0}\rightarrow 0$ for some $f\in C^\infty(M,\mathbb{R}_+)$, we have $\widetilde c_k (M,f_j\alpha)\rightarrow \widetilde c_k (M,f\alpha)$.
    \vspace{0.2cm}
    \item[-] (Spectrality) For any $k$, if the number $c_k(\alpha)$ is finite, then it is a linear combination, with non-negative integer coefficients, of elements in $\mathcal{A}(M,\alpha)$. 
\end{itemize}
\vspace{0.2cm}
Given a contact form such that $\xi=\ker \alpha$ has torsion first Chern class $c_1(\xi)$ (which is equal to the Euler class of $\xi$ as a plane subbundle of $TM$), for each $k$ we have $c_k(M,\alpha)<\infty$, and some of them are non-zero, see \cite[Remark 3.2]{Hut}. Indeed, the group $\ech(M,\alpha, \Gamma)$ is infinitely generated for any $\Gamma$ such that $c_1(\xi)+2 \operatorname{PD}[\Gamma]\in H^2(M,\mathbb{Z})$ is torsion (here $\operatorname{PD}[\Gamma]$ denotes the Poincaré dual of $\Gamma$). This follows from the isomorphism between ECH and Seiberg-Witten Floer cohomology proved by Taubes \cite{Tau}. In particular, if the first Chern class of $\xi$ is torsion then $\ech(M,\xi,0)$ is infinitely generated, and hence it cannot happen that every spectral invariant vanishes.
 \begin{Remark}\label{rem:SW}
 The fact that $c_1(\xi)$ is torsion only depends on the homotopy class of plane fields defined by $\xi$. Thus, in our context, if a certain vorticity vector field $X$ is contact type with associated contact structure $\xi$, the fact that $\xi$ has torsion first Chern class only depends on the homotopy class $a$ of the vector field $X$, which determines the homotopy class $p_a$ of $\xi$.
\end{Remark}

\begin{Remark}\label{rem:SWCANON} Let us point out as well that more than Equation \eqref{eq:ECHiso1} holds: for any two contact forms $\alpha_1, \alpha_2$ that define isotopic contact structures $\xi_1, \xi_2$ and any $\Gamma$, the embedded contact homology groups are canonically isomorphic. This holds from the fact that the isomorphism between ECH and Seiberg-Witten Floer cohomology is canonical as shown by Taubes, and the Seiberg-Witten Floer cohomology is a topological invariant that does not depend on the specific contact form defining any contact structure in $[\xi]$. Thus, in particular, one can write $\ech(M, [\xi], 0)$, meaning that via the isomorphism with Seiberg-Witten Floer cohomology, one can unambiguously identify all the ECH groups for any contact form defining a contact structure in $[\xi]$.
\end{Remark}
\subsection{Spectral invariants of contact type solutions and continuity}

 The following lemma, which is a direct consequence of \cite[Lemma 3.9]{Hut}, is key for our purposes, as it shows that when $\Gamma=0$, spectral invariants of a contact primitive of a contact type two-form does not depend on the specific choice of primitive.
\begin{lemma}\label{lem:spectwoform}
Let $d\alpha$ be a contact type two-form in $M$. Let $\alpha$ and $\alpha'$ be two contact primitives of $d\alpha$. Then $\widetilde c_k(\alpha)=\widetilde c_k(\alpha')$ for any positive integer $k$.
\end{lemma}
Hence, the full ECH spectrum is an invariant of the contact type two-form. This is an important property in the setting of symplectic embeddings, where ECH spectral invariants can be used to define symplectic capacities of domains with contact boundary. Definition \ref{def:spectral} and Lemma \ref{lem:spectwoform} show that the following notion is well-defined.
\begin{defi}\label{spectral}
    Let $u_0$ be a contact type initial condition to the Euler equations, i.e. $\curl u_0$ is a contact type vector field. Let $\alpha$ be any contact primitive of $d\lambda$, where $\lambda=g(u_0,\cdot)$. For any positive integer $k$, the spectral invariant $c_k(u_0)$ is defined as $c_k(u_0):=\widetilde c_k(\alpha)$.
\end{defi}
Hence, each spectral invariant defines a functional 
$$c_k: \mathcal{V}_{[\xi]}\longrightarrow \mathbb{R}. $$
The main relevance of these functionals is the following theorem, which establishes that they are conserved quantities of the Euler equations which are continuous in suitable topologies.

\begin{theorem}\label{thm:contspec}
   Let $\mathcal{V}_{[\xi]}\subset \mathcal{V}_{a,h,e}$ be the $C^1$-open set of $\mathcal{X}_\mu(M)$ invariant by the Euler equations defined by an isotopy class of contact structures $[\xi]$ with torsion first Chern class. For each non-vanishing positive integer $k$, the functional
\begin{align*}
     c_k: \mathcal{V}_{[\xi]} &\longrightarrow \mathbb{R}_{\geq 0}\\
                u &\longmapsto c_k(u)
\end{align*}
is continuous with respect to the $C^{1,s}$-topology in $\mathcal{V}_{[\xi]}$ for each $s\in (0,1]$. Furthermore, it is a first integral of the Euler equations in $\mathcal{V}_{[\xi]}$.
\end{theorem}

\begin{proof}
Let $(M,g)$ be a Riemannian three-manifold and $u$ be a divergence-free vector field in $\mathcal{V}_{[\xi]}$. Let $v \in \mathcal{V}_{[\xi]}$ be another vector field satisfying $||v-u||_{C^{1, s}} \leq \epsilon$. The corresponding dual 1-forms $\lambda_{v}$ and $\lambda_{u}$ are $C^{1,s}$-close as well, and thus their contact type two-forms satisfy $||d\lambda_{v}-d\lambda_{u}||_{C^{0, s}} \leq C \epsilon$. By mimicking the proof of the first item in Lemma \ref{lem:Xcont}, but applying in it Lemma \ref{lem:C1bound} instead of Lemma \ref{lem:C0bound}, we find primitive contact one-forms $\alpha_{v}$ and $\alpha_{u}$ of $d \lambda_{v}$ and $d \lambda_{u}$ that verify $||\alpha_{v}-\alpha_{u}||_{C^{1, s}} \leq C \epsilon$.

Because of the $C^{1}$-closeness, the $1$-forms in the segment $\alpha_{t}=\alpha_{u}+t(\alpha_{v}-\alpha_{u}),\,\, t \in[0, 1]$, are all contact as well. We define a non-autonomous vector field $X_{t}$ by setting
\[
i_{X_{t}} (d \alpha_{t} \wedge \alpha_{t}):=(\alpha_{u}-\alpha_{v}) \wedge \alpha_{t}=-\partial_{t} \alpha_{t} \wedge \alpha_t \,.
\]
In what follows we will reproduce the standard proof of Gray's stability theorem, but keeping track of the norms of the different objects appearing through the argument. To begin with, observe that the $C^{1, s}$-closeness of $\alpha_{u}$ and $\alpha_{v}$ implies that $||X_{t}||_{C^{1, s}} \leq C \epsilon$. In addition, we have that
\[
i_{X_t} d \alpha_{t}=-\pp{\alpha_t}{t}+f_{t} \alpha_{t},
\]
for some $t$-dependent function $f_{t}$, whose $C^{0}$ norm is easily seen to be small. Indeed, denote by $R_t$ the Reeb field of $\alpha_t$, so that we have $f_t=\partial_{t} \alpha_t (R_t)$. Since $\partial_{t} \alpha_t=\alpha_{v}-\alpha_{u}$ has $C^{0}$ norm bounded by $\epsilon$, we conclude that $||f_t||_{C^{0}} \leq C' \epsilon$, where the constant depends on the $C^{0}$ norm of $R_t$, which, provided $\epsilon$ is small, can be bounded by, for example, twice the one of $R_0$.

This understood, we have
\begin{align*}
\frac{d}{dt} (\phi^{t}_{X_{t}})^{*}\alpha_{t}&=(\phi^{t}_{X_{t}})^{*} (\partial_{t} \alpha_{t}+i_{X_t} d \alpha_{t})\\
&=(\phi^{t}_{X_{t}})^{*}(f_{t} \alpha_{t})=(f_{t} \circ \phi^{t}_{X_{t}}) (\phi^{t}_{X_{t}})^{*}  \alpha_{t},
\end{align*}
that is, the one-parameter family of 1-forms $\beta_{t}:=(\phi^{t}_{X_{t}})^{*} \alpha_{t}$ satisfies the linear ODE
\[
\frac{d}{dt} \beta_{t}= g_{t} \beta_{t},
\]
with $g_t:=f_{t} \circ \phi^{t}_{X_{t}}$. Therefore, at any point $p$ in the manifold,
\[
\beta_{t}(p)= e^{\int_{0}^{s} g_{s}(p) ds} \beta_{0}(p),
\]
and in particular, setting $\Phi:=\phi^{1}_{X_{1}}$ we have
\[
\Phi^{*} \alpha_{v}= h\alpha_{u},
\]
with
\[
h(p)=e^{\int_{0}^{1} f_{s}(\phi^{s}_{X_{s}}(p)) ds} .
\]
Since $X_{t}$ and $f_t$ are $C^{1, s}$ and $C^{0}$-small respectively, this function is close to $1$, $||h-1||_{C^{0}} \leq C' \epsilon$. 

Consider a sequence of vectors $\{v_{i}\}_{i=1}^{\infty}$ with $||v_{i}-u||_{C^{1, s}} =\epsilon_{i} \rightarrow 0$, so that, by our previous discussion, $\Phi_{i}^{*} \alpha_{v_i}=h_{i} \alpha_u$ with $||h_{i}-1||_{C^{0}} \rightarrow 0$. The $C^{0}$-continuity property of the ECH spectral invariants (see Definition \ref{def:spectral}) yields 
\[
\widetilde c_k(\Phi^{*}_{i} \alpha_{v_i})=\widetilde c_k(h_{i} \alpha_u) \longrightarrow \widetilde c_k(\alpha_{u})\,.
\]
But the spectral invariants are invariant under diffeomorphisms, $\widetilde c_k(\Phi^{*}_{i} \alpha_{v_i})=\widetilde c_k(\alpha_{v_i})$, so we conclude that $\widetilde c_k( \alpha_{v_i}) \rightarrow  \widetilde c_k(\alpha_{u})$, and thus  $c_k(v_{i})\rightarrow c_{k}(u)$, as claimed. \\

Let us now prove that $c_k$ is preserved by the flow defined by the Euler equations, i.e. its value is constant along an orbit of the flow. Let $u_t$ be a contact type solution to the Euler equations in $(M,g)$, with $u_t \in \mathcal{V}_{[\xi]}$ and denote by $\lambda_0=g(u_0,\cdot)$. Let $\alpha_0$ be a contact primitive of $d\lambda_0$. Transport of vorticity (Equation \eqref{eq:transport}) ensures that the two-form $d\lambda_t$, where $\lambda_t=g(u_t,\cdot)$, admits as primitive the contact form $\alpha_t=(\phi_t^{-1})^*\alpha_0$. It follows from the definition of the spectral invariants that diffeomorphic contact forms have the same spectral invariants, i.e. $\widetilde c_k(\alpha_t)=\widetilde c_k(\alpha_0)$. In particular
$$c_k (u_t)=\widetilde c_k(\alpha_t)=\widetilde c_k(\alpha_0)=c_k(u_0),$$
for all $t$ for which the solution is defined. Accordingly the value of the functional $c_k: \mathcal{V}_{[\xi]}\rightarrow \mathbb{R}$ is constant along the orbit of $u_0$.
\end{proof}

\begin{Remark}
The argument in the proof does not apply if one wants to show that the spectral invariants define $C^1$-continuous integrals of the Euler equations. This is because one can only construct in general $C^0$-close contact primitives of the $C^1$-close initial conditions. As we saw in Lemma \ref{lem:Xcont} these $C^0$-close contact primitives define isotopic contact structures, but it is not clear at all that the isotopy identifying both contact structures is small in $C^1$-topology (or even $C^0$), which is a key step in the proof of Theorem \ref{thm:contspec}. On the other hand, when there is an initial condition $u_0$ such that $g(u_0,\cdot)=\alpha$ is a contact form, then close to that initial condition the spectral invariants are $C^1$-continuous since contracting with the metric gives $C^1$-close contact forms and the proof applies in this case. This happens for example near a rotational Beltrami stationary solution, see Remark \ref{rem:rotBelt}. Nevertheless, our techniques to construct initial conditions with prescribed energy and vorticity in the adjoint orbit of a given contact type exact divergence-free vector field do not necessarily produce initial conditions with the property that $\alpha=g(u_0,\cdot)$ is precisely a contact primitive of $d\alpha$.
\end{Remark}

\begin{Remark}
Most of our discussion in this section applies to other types of spectral invariants of contact manifolds, see e.g. \cite{Hut3, Hut4}, as long as the value of the invariant only depends on the differential of the contact form.  As we have seen, this is the case for several of the ECH spectral invariants, making them suitable to define continuous integrals. Furthermore, using the ECH spectral invariants one can always find a continuous integral that is non-trivial, as we explain in the next section.
\end{Remark}

\subsection{Invariant $C^{1,s}$-open sets}

In this section, we will show that the spectral invariants introduced in the previous section can be used to find countably many $C^{1,s}$-open sets of divergence-free vector fields that do not mix under the Euler equations, whenever it is possible to ensure the non-triviality of any of the integrals introduced in the previous section. As we mentioned earlier, the non-triviality of $\ech(M,\xi,0)$ always holds when $c_1(\xi)$ is torsion, so we restrict to this case for our statement.

\begin{theorem}\label{thm:C1smixing}
    Fix an homotopy class $a$, an helicity $h$, and an isotopy class of contact structure $[\xi]$ on $M$ such that $c_1(\xi)$ is torsion. For any large enough value of the energy $e$, there are countably many $C^{1,s}$-open sets $U_i$ in $\mathcal{V}_{[\xi]}$  such that for any $i\neq j$, we have
    $$ \varphi_t(U_i)\cap U_j=\emptyset, $$
    for every $t$ for which the flow of the Euler equations is defined.    
\end{theorem}
\begin{Remark}\label{rem:chernhyp}
   The hypothesis that $c_1(\xi)$ is torsion is, for instance, always satisfied in rational homology spheres. In general, $c_1(\xi)$ only depends on the homotopy class of plane fields defined by $\xi$, and hence on the homotopy class of the vorticity field $a$. On any closed three-manifold, there are countably many homotopy classes of plane fields for which any contact structure in that homotopy class has torsion first Chern class, and thus our Theorem applies to countably many choices of $a$. This is seen by the fact that there are countably many homotopy classes of plane fields for which the Euler class (which is equal to the first Chern class) vanishes. We refer to \cite[Sections 4.2 and 4.3]{Ge} for a very detailed discussion of these facts.
\end{Remark} 

This theorem follows essentially from the continuity of the spectral invariants established in Theorem \ref{thm:contspec} and the following Lemma, which establishes that it is always possible to find spectral invariants that are non-trivial.

\begin{lemma}\label{lem:spectralinterval}
    Let $(M,\xi)$ be a closed contact three-manifold. Let $\alpha$ be any contact form. For any $k$ such that $\widetilde c_k(\alpha)=N >0$ is finite and any $r>N$, there is family of contact forms $\alpha_t$ defining $\xi$ such that:
    \begin{itemize}
        \item[-] their contact volume is constant $\int_M \alpha_t\wedge d\alpha_t=\int_M \alpha\wedge d\alpha$,
        \item[-] the image of $\widetilde c_k(\alpha_t)$ contains the open interval of $(N,r)$.
    \end{itemize}
\end{lemma}
\begin{proof}
    Let $\alpha$ be any contact form defining $\xi$ with contact volume $\int_M \alpha\wedge d\alpha=h$, and $\widetilde c_k(\alpha)$ some finitely value spectral invariant $\widetilde c_k(M,\alpha)=N > 0$. 

    Recall that the systolic ratio of a contact form $\gamma$ is 
    $$\rho(\gamma)=\frac{T_{\operatorname{min}}(\gamma)}{\operatorname{vol}(\gamma)}, $$
    where $T_{\operatorname{min}}(\gamma)$ denotes the minimal action of a periodic orbit of the Reeb flow of $\gamma$, and $\operatorname{vol}(\gamma)=\int_M \gamma \wedge d\gamma$. By the main theorem in \cite{ABHS}, for any arbitrarily large $r>0$ there exists a contact form $\alpha_1$ whose systolic ratio is larger than $\frac{r}{h}$. The systolic ratio is invariant under constant rescaling of the contact form, so up to multiplying $\alpha_1$ by a positive constant $c$, we can assume that $\operatorname{vol}(\alpha_1)=h$. In particular the minimal action of a period Reeb orbit of $\alpha_1$ satisfies 
    $$T_{\operatorname{min}}(\alpha_1)>r.$$ 
    Write $\alpha_1=f\alpha$, where $f\in C^\infty(M)$ everywhere positive. We know that $\int_M f\alpha\wedge d\alpha=\int_M \alpha\wedge d\alpha$, so the path of contact forms $\alpha_t= t \alpha_1 + (1-t)\alpha= [tf+(1-t)]\alpha$ is a path of contact forms with constant contact volume $h$. By continuity of the spectral invariant and its spectrality, we know that $\widetilde c_k(\alpha_t)$ is positive and a continuous function of $t$. Furthermore, again by spectrality, we know as well that $\widetilde c_k(\alpha_1)>r$ necessarily. It follows that $[N,r]\subset \{\widetilde c_k(\alpha_t)\,, t \in [0, 1]\}$, as claimed.
\end{proof}

We conclude by establishing the non-mixing properties of contact type solutions.

\begin{proof}[Proof of Theorem \ref{thm:C1smixing}]
Fix $a,h$ and an isotopy class of contact structures $[\xi]$ on $M$, which by assumption has torsion first Chern class (equivalently, torsion Euler class). As discussed in section \ref{ss:ECH}, there is some $k$ for which $\widetilde c_k$ is non-trivial on some contact form defining $\xi$. We can apply Lemma \ref{lem:spectralinterval} to find a family of contact forms $\alpha_s$ with $\ker \alpha_s=\xi$, with contact volume $h$, and such that $\widetilde c_k(\alpha_s)$ contains an open interval $(a,b)$. Denote by $X_s$ the exact divergence-free vector fields defined by 
$$\iota_{X_s}\mu=d\alpha_s.$$ 
An application of Theorem \ref{thm:localvortex} to each $X_s$ shows that there is a (possibly non-continuous) family of positive real numbers $e_s$ such that for any $e>e_s$ there is a volume-preserving vector field $u_s$ such that:
\begin{itemize}
    \item[-] $\curl(u_s)=Y_s$ where $Y_s=\varphi_s^*X_s$ for a (non-continuous) family of volume-preserving diffeomorphisms $\varphi_s: M\rightarrow M$ isotopic to the identity,
    \item[-] $\mathcal{E}(u_s)=e$.
\end{itemize} 
The family $\alpha_s$ is parametrized by the compact set $[0,1]$, so let $E$ be any positive real greater than $e_s$ for each $s$. Then the family $u_s$ is a (possibly non-continuous) family of divergence-free vector fields in $\mathcal{V}_{[\xi]}\subset \mathcal{V}_{a,h, e}$. Furthermore, the spectral invariant of $dg(u_s,\cdot)=d\varphi_s^*\alpha_s$ has the same value as the spectral invariant of $\alpha_s$, so we deduce that $c_k(u_s)$ contains the interval $(a,b)$. Choose countably many disjoint points $x_i\in (a,b)$ and let $u_{s_i}$ be a divergence-free vector field such that $c_k(u_{s_i})=x_i$. By the continuity of the spectral invariants proved in Theorem \ref{thm:contspec}, we can now find for each $u_{s_i}$ a $C^{1,s}$-open set $U_i\subset \mathcal{V}_{[\xi]}$ such that the intervals $\operatorname{Im}_{c_k}(U_i)$ are pairwise disjoint. Since the value of the spectral invariant is invariant under the Euler equations, the conclusion of the Theorem follows. 
\end{proof}

Theorem \ref{thm:spec} stated in the introduction is a combination of Theorem \ref{thm:contspec} and Theorem \ref{thm:C1smixing}.

\end{document}